\documentclass{amsart}

\usepackage{amssymb}
\usepackage{amsthm}
\usepackage{mathtools}
\usepackage{mathrsfs} 
\usepackage[normalem]{ulem} 
\usepackage{float} 
\usepackage{caption} 
\usepackage{enumerate} 
\usepackage{wasysym} 
\usepackage{enumitem}
\usepackage{multirow} 
\usepackage{hyperref} 
\hypersetup{colorlinks=false, citecolor=blue, linkcolor=blue}
\usepackage{xparse} 
\usepackage{xargs}
\usepackage{todonotes} 
\usepackage{ifthen} 
\usepackage{color}

\usepackage{tikz}
\usetikzlibrary{matrix, automata, arrows, calc}
\usetikzlibrary{decorations}
\usetikzlibrary{decorations.markings}

\AtBeginDocument{}

\AtBeginDocument{%
        \LetLtxMacro\oldref{\ref}%
        \DeclareRobustCommand{\ref}[2][]{(\oldref#1{#2})}%
}

\usepackage{aliascnt}
\newtheoremstyle{thmstyle}{}{}{\itshape}{}{\bfseries}{.}{5pt}{}
\newtheoremstyle{exstyle}{}{}{}{}{\bfseries}{.}{5pt}{}
\newtheoremstyle{defstyle}{}{}{}{}{\bfseries}{.}{5pt}{}
\newtheoremstyle{remstyle}{}{}{}{}{\bfseries}{.}{5pt}{}
\newtheoremstyle{proofstyle}{}{}{}{}{\bfseries}{.}{5pt}{}

\theoremstyle{thmstyle}
\newtheorem{thm}{Theorem}[section]
\newtheorem{theorem}[thm]{Theorem}
\def\thmautorefname~#1\null{Theorem~#1\null}

\newaliascnt{lemma}{thm}
\newtheorem{lemma}[lemma]{Lemma}
\aliascntresetthe{lemma}
\def\lemmaautorefname~#1\null{Lemma~#1\null}

\newaliascnt{proposition}{thm}
\newtheorem{proposition}[proposition]{Proposition}
\aliascntresetthe{proposition}
\def\propositionautorefname~#1\null{Proposition~#1\null}

\newaliascnt{corollary}{thm}
\newtheorem{corollary}[corollary]{Corollary}
\aliascntresetthe{corollary}
\def\corollaryautorefname~#1\null{Corollary~#1\null}

\newaliascnt{conjecture}{thm}
\newtheorem{conjecture}[conjecture]{Conjecture}
\aliascntresetthe{conjecture}
\def\conjectureautorefname~#1\null{Conjecture~#1\null}

\def\equationautorefname~#1\null{\text{(#1)\null}}
\def\figureautorefname~#1\null{Figure~#1\null}
\def\itemautorefname~#1\null{(#1)\null}
\def\sectionautorefname~#1\null{Section~#1\null}
\def\subsectionautorefname~#1\null{Subsection~#1\null}

\theoremstyle{exstyle}
\newaliascnt{example}{thm}
\newtheorem{example}[example]{Example}
\aliascntresetthe{example}
\def\exampleautorefname~#1\null{Example~#1\null}

\newaliascnt{remark}{thm}
\newtheorem{remark}[remark]{Remark}
\aliascntresetthe{remark}
\def\remarkautorefname~#1\null{Remark~#1\null}

\theoremstyle{defstyle}
\newaliascnt{defn}{thm}
\newtheorem{defn}[defn]{Definition}
\aliascntresetthe{defn}
\def\defnautorefname~#1\null{Definition~#1\null}

\newtheoremstyle{named}{}{}{\itshape}{}{\bfseries}{.}{.5em}{\thmnote{#3}}
\theoremstyle{named}

\definecolor{ablue}{RGB}{0,39,236}
\newcommand{\definition}[1]{\textcolor{ablue}{\textit{#1}}}

\graphicspath{{./}}

\usepackage[bottom]{footmisc}

\newcommand{\st}{\mid}
\newcommand{\set}[2]{\left\{#1 \st #2 \right\}} 
\newcommand{\biggset}[2]{\bigg\{ #1 \;\bigg|\; #2 \bigg\}} 
\newcommand{\gen}[1]{\left< #1 \right>}
\newcommand{\coloneqq}{\mbox{\,\raisebox{0.2ex}{\scriptsize\ensuremath{\mathrm:}}\ensuremath{=}\,}} 
\newcommand{\setm}{\smallsetminus} 

\newcommand{\R}{\mathbb{R}} 
\newcommand{\Z}{\mathbb{Z}} 
\newcommand{\Sbb}{\mathbb{S}} 

\newcommand{\meet}{\wedge} 
\newcommand{\join}{\vee} 
\newcommand{\cover}{\lessdot} 
\DeclareMathOperator{\ji}{JIrr}
\DeclareMathOperator{\mi}{MIrr}
\DeclareMathOperator{\op}{op}

\newcommand{\arrangement}{\mathcal{A}} 
\newcommandx{\regions}[1][1=\arrangement]{\mathscr{R}_{\!#1}} 
\newcommandx{\faces}[1][1=\arrangement]{\mathscr{F}_{\!\!#1}} 
\newcommandx{\zonotope}[1][1=\arrangement]{\mathbf{Z}_{\!#1}} 
\newcommand{\boundary}{\mathscr{B}} 
\DeclareMathOperator{\rank}{rank} 
\newcommand{\rootInversionSet}{\mathbf{R}} 
\newcommand{\intersection}{\mathcal{I}} 
\newcommand{\covectors}{\mathcal{L}} 
\DeclareMathOperator{\vspan}{span} 
\DeclareMathOperator{\cone}{cone} 
\let\int\relax
\DeclareMathOperator{\int}{int} 
\DeclareMathOperator{\codim}{codim} 
\DeclareMathOperator{\pricone}{\mathbf{C}} 
\DeclareMathOperator{\norcone}{\mathbf{C}^\diamond} 

\DeclareMathOperator{\PR}{PR} 
\newcommandx{\pr}[2][1=\arrangement, 2=B]{\PR \! \left( #1,#2 \right)} 
\newcommand{\prle}{\le_{\PR}} 
\newcommand{\prcover}{\cover_{\PR}} 
\newcommand{\prmeet}{\meet_{\PR}} 
\newcommand{\prjoin}{\join_{\PR}} 
\DeclareMathOperator{\FW}{FW} 
\newcommandx{\fw}[2][1=\arrangement, 2=B]{\FW \! \left( #1,#2 \right)} 
\newcommand{\fwle}{\le_{\FW}} 
\newcommand{\fwl}{<_{\FW}} 
\newcommand{\fwcover}{\cover_{\FW}} 
\newcommand{\fwmeet}{\meet_{\FW}} 
\newcommand{\fwjoin}{\join_{\FW}} 
\newcommand{\rwo}{\le_{R}} 
\newcommand{\cole}{\le_\covectors} 
\newcommand{\col}{<_\covectors} 

\DeclareMathOperator{\del}{del} 
\DeclareMathOperator{\lk}{lk} 
\DeclareMathOperator{\rk}{rk} 
\DeclareMathOperator{\susp}{susp} 

\newcommand{\ie}{\textit{i.e.},~} 
\newcommand{\eg}{\textit{e.g.}~} 

\newcommand{\iso}{\cong} 

\DeclarePairedDelimiter\abs{\lvert}{\rvert}%
\DeclarePairedDelimiter\norm{\lVert}{\rVert}%

\makeatletter
\def\l@section{\@tocline{1}{3pt}{0pc}{}{}}
\makeatother
\let\oldtocpart=\tocpart
\renewcommand{\tocpart}[2]{\bf\large\oldtocpart{#1}{#2}}
\let\oldtocsection=\tocsection
\renewcommand{\tocsection}[2]{\bf\oldtocsection{#1}{#2}}

\title[The Facial Weak Order on Hyperplane Arrangements]{The Facial Weak Order \\ on Hyperplane Arrangements}
\thanks{AD was supported by a FRQNT scholarship. CH was supported by NSERC Discovery grant {\em geometric and algebraic combinatorics of Coxeter groups}. VP was supported by French ANR grants SC3A~(15\,CE40\,0004\,01) and CAPPS~(17\,CE40\,0018).}

\author[A.~Dermenjian]{Aram Dermenjian} 
\address[A.~Dermenjian]{LaCIM, Universit\'e du Qu\'ebec \`A Montr\'eal}
\email{aram.dermenjian@gmail.com}
\urladdr{http://dermenjian.com}

\author[C.~Hohlweg]{Christophe Hohlweg} 
\address[C.~Hohlweg]{LaCIM, Universit\'e du Qu\'ebec \`A Montr\'eal}
\email{hohlweg.christophe@uqam.ca}
\urladdr{http://hohlweg.math.uqam.ca/en/about-this-site/}

\author[T.~McConville]{Thomas McConville} 
\address[T.~McConville]{Massachusetts Institute of Technology, Cambridge}
\email{thomasmc@mit.edu}
\urladdr{http://math.mit.edu/~thomasmc/}

\author[V.~Pilaud]{Vincent Pilaud} 
\address[V.~Pilaud]{CNRS \& LIX, \'Ecole Polytechnique, Palaiseau}
\email{vincent.pilaud@lix.polytechnique.fr}
\urladdr{http://www.lix.polytechnique.fr/~pilaud/}

\begin{document}

\begin{abstract}
    We extend the facial weak order from finite Coxeter groups to central hyperplane arrangements.
    The facial weak order extends the poset of regions of a hyperplane arrangement to all its faces.
    We provide four non-trivially equivalent definitions of the facial weak order of a central arrangement: (1)~by exploiting the fact that the faces are intervals in the poset of regions, (2)~by describing its cover relations, (3)~using covectors of the corresponding oriented matroid, and (4)~using certain sets of normal vectors closely related to the geometry of the corresponding zonotope.
    Using these equivalent descriptions, we show that when the poset of regions is a lattice, the facial weak order is a lattice.
    In the case of simplicial arrangements, we further show that this lattice is semidistributive and give a description of its join-irreducible elements.
    Finally, we determine the homotopy type of all intervals in the facial weak order.
\end{abstract}

\vspace*{-.4cm}
\maketitle

\tableofcontents


\section{Introduction}

A \definition{hyperplane arrangement} is a finite collection $\arrangement$ of linear hyperplanes in a finite dimensional real vector space $V$.
Its \definition{regions} are the closures of the connected components of the complement in $V$ of the union of all hyperplanes in $\arrangement$.
A region is \definition{simplicial} if the normal vectors to its bounding hyperplanes are linearly independent, and the arrangement is \definition{simplicial} if all its regions are.
The \definition{zonotope} of the arrangement~$\arrangement$ is a convex polytope dual to the arrangement~$\arrangement$, obtained as the Minkowski sum of line segments normal to the hyperplanes of~$\arrangement$.

The regions of a hyperplane arrangement~$\arrangement$ can be ordered as follows.
Define the \definition{separation set}~$S(R,R')$ between two regions~$R$ and~$R'$ of~$\arrangement$ as the set of hyperplanes of~$\arrangement$ separating the two regions~$R$ and~$R'$.
For a fixed base region $B$, the \definition{poset of regions} is the set of regions of~$\arrangement$ ordered by inclusion of their separation sets~$S(B,R)$ with the base region~$B$.
A.~Bj\"orner, P.~H.~Edelman and G.~M.~Ziegler \cite{BjornerEdelmanZiegler} showed that the poset of regions is a lattice if $\arrangement$ is simplicial, and that the base region~$B$ is simplicial if the poset of regions is a lattice.
The Hasse diagram of the poset of regions can also be seen as the graph of the zonotope of~$\arrangement$, oriented from the base region~$B$ to its opposite region~$-B$.

In this paper, we extend the study of the \definition{facial weak order} $\fw$, as introduced in \cite{DermenjianHohlwegPilaud} for Coxeter arrangements.
This order is a poset structure on the faces of the hyperplane arrangement~$\arrangement$ or, equivalently, of the zonotope of~$\arrangement$.
It was first introduced by D.~Krob, M.~Latapy, J.-C.~Novelli, H.-D.~Phan, and S.~Schwer in \cite{KrobLatapyNovelliPhanSchwer} for the braid arrangement (the Coxeter arrangement of type $A$) where it was shown to be a lattice.
It was then extended to arbitrary Coxeter arrangements by P.~Palacios and M.~Ronco in~\cite{PalaciosRonco} and it was shown to be a lattice for arbitrary Coxeter arrangements in \cite{DermenjianHohlwegPilaud}.
The aims of this article are to extend the facial weak order to central hyperplane arrangements.

The first part of this article, contained in \autoref{sec:Facial_Weak_Order_on_the_poset_of_regions} and \autoref{sec:Geometric_interpretations}, is dedicated to providing four equivalent definitions for the facial weak order on a given central hyperplane arrangement:
\begin{itemize}
    \item in terms of \definition{separation set} comparisons between the minimal and maximal regions incident to a face (\autoref{subsec:facial_weak_order}),
    \item by providing a precise description of its \definition{covering relations} (\autoref{subsec:cover_relations}),
    \item in terms of \definition{covectors} of the associated oriented matroid (\autoref{subsec:covectors}),
    \item and in terms of \definition{root inversion sets} of the normals to the hyperplanes of the arrangement (\autoref{subsec:root_inversion_sets}), closely related to the geometry of the corresponding zonotope (\autoref{subsec:zonotopes}).
\end{itemize}
We prove these four definitions to be equivalent in~\autoref{thm:equivalence1} and \autoref{thm:equivalence2}.
In the case of a Coxeter arrangement, this recovers the descriptions in~\cite{DermenjianHohlwegPilaud}.

In \autoref{sec:lattice}, we show that if the poset of regions of a hyperplane arrangement is a lattice, then the facial weak order is a lattice (\autoref{thm:simplicial_is_lattice}).
This is achieved using the BEZ lemma \cite[Lemma 2.1]{BjornerEdelmanZiegler} which states that a poset is a lattice as soon as there exists a join~$x \join y$ for every two elements $x$ and $y$ that both cover the same element.
This extends the results of~\cite{KrobLatapyNovelliPhanSchwer} for the braid arrangement and of~\cite{DermenjianHohlwegPilaud} for Coxeter arrangements.

For a general arrangement $\arrangement$, the facial weak order may not be a lattice, but its topology still admits a nice description that we study in \autoref{sec:topology}.
There are a wide variety of simplicial complexes associated to a hyperplane arrangement.
Typically, complexes that depend on the matroid structure of $\arrangement$ are homotopy equivalent to a wedge of (several) spheres, \eg the independence complex, the reduced broken circuit complex, or the lattice of flats \cite{Bjorner_homologyMatroid}.
On the other hand, complexes that depend on the \emph{oriented} matroid structure of $\arrangement$ tend to be homotopy equivalent to a single sphere or are contractible, \eg the complexes of acyclic, convex, or free sets \cite{EdelmanReinerWelker}, the poset of regions \cite{Edelman}, or the poset of cellular strings \cite{Bjorner_EssentialChains}.
We compute the homotopy types of intervals of the facial weak order (\autoref{thm_weak_top}).
Keeping with the aforementioned trends, we prove that every interval of the facial weak order is either contractible or homotopy equivalent to a sphere.

To conclude, let us mention two directions that are not explicitly explored here to keep the paper short.
First, although we use the language of oriented matroids, we only deal here with facial weak order of hyperplane arrangements.
The results presented here seem however to extend in the context of simple simplicial oriented matroids.
Second, using the same tools as in~\cite{DermenjianHohlwegPilaud}, one can observe that when the arrangement is simplicial, each lattice congruence of the poset of regions naturally translates to a lattice congruence of the facial weak order.


\section{Facial weak order on the poset of regions}
\label{sec:Facial_Weak_Order_on_the_poset_of_regions}

In this section we start by recalling classical definitions on hyperplane arrangements.
For more details, we refer the reader to  the book by P.~Orlik and H.~Terao~\cite{OrlikTerao}, the book by R.~Stanley \cite{Stanley_EnumerativeCombinatorics} and the paper by A.~Bj\"orner, P.~H.~Edelman and G.~M.~Ziegler \cite{BjornerEdelmanZiegler}.
We then introduce the \definition{facial weak order} and discuss its cover relations.


\subsection{Hyperplane arrangements}

Let $(V, \gen{\cdot, \cdot})$ be an $n$-dimensional real Euclidean vector space.
A \definition{central hyperplane arrangement}, or \definition{arrangement} for short, is a finite set~$\arrangement$ of linear hyperplanes in~$V$.
For each $H \in \arrangement$, we choose some fixed nonzero vector~$e_H$  normal to~$H$, that is, such that~$H = \set{v \in V}{\gen{e_H, v} = 0}$ (the choice of the normal vectors $e_H$ is unique up to nonzero scalar multiplication).
We also consider the two half spaces $H^+ \coloneqq \set{v \in V}{\gen{e_H, v} > 0}$ and~${H^- \coloneqq \set{v \in V}{\gen{e_H, v} < 0}}$ bounded by~$H$.

The \definition{rank} of $\arrangement$ is the dimension~$\rank(\arrangement)$ of the linear subspace~$V'$ spanned by the vectors~$e_H$, for~$H \in \arrangement$.
An arrangement $\arrangement$ is \definition{essential} if ${\rank(\arrangement) = \dim(V)}$, or equivalently if the intersection of all hyperplanes of~$\arrangement$ is the origin.
We assume our arrangements to be essential unless stated otherwise.
From a combinatorial perspective the specialization to essential arrangements causes no loss of generality.
This is due to the fact that for each arrangement $\arrangement$ of rank $m$ in $V \iso \R^n$ there is an associated essential arrangement $\arrangement'$ in $V' \iso \R^m$ whose face structure is similar.
See \cite[Section 2.1]{BjornerLasVergnasSturmfelsWhiteZiegler} for more details.

The \definition{regions} of an arrangement~$\arrangement$ are the closures of the connected components of $V \setm \left(\bigcup_{H \in \arrangement} H\right)$.
We denote by~$\regions$ the set of regions of~$\arrangement$.
A \definition{wall} of a region~$R$ in~$\arrangement$ is a bounding hyperplane~$H \in \arrangement$ of $R$, that is, $\dim(H \cap R) = \dim(V)-1$.
A region $R$ is said to be \definition{simplicial} if the normal vectors of its walls are linearly independent.
If $\arrangement$ is essential then a region is simplicial if and only if it has precisely $\rank(\arrangement)$ walls.
An arrangement is \definition{simplicial} if all its regions are simplicial.

A \definition{face} of~$\arrangement$ is the intersection of some regions of~$\arrangement$.
We denote by $\faces$ the set of faces of~$\arrangement$.
Note that the regions are the codimension $0$ faces of~$\arrangement$.
The \definition{face poset} of the arrangement~$\arrangement$ is the poset~$(\faces, \subseteq)$ of faces of~$\arrangement$ ordered by inclusion.
The \definition{face lattice} of the arrangement~$\arrangement$ is the face poset together with the vector space itself as the maximum element.
In this paper, we will consider a different poset structure on~$\faces$.

\begin{example}
    \label{ex:Coxeter1}
    Well-known examples of simplicial hyperplane arrangements are the \definition{Coxeter arrangements}.
    These are the hyperplane arrangements associated to a Coxeter system $(W,S)$.
    See \autoref{fig:CoxeterArrangements} for an illustration of the Coxeter arrangements of types~$A_3$, $B_3$ and~$H_3$.
    We refer the reader to the books~\cite{Humphreys, BjornerBrenti} for comprehensive surveys on Coxeter groups.
    \autoref{fig:A2} gives an example of the type~$A_2$ Coxeter arrangement together with its faces.
    The $R_i$ (in blue) are the six regions of the arrangement and are the codimension $0$ faces.
    There are also six codimension~$1$ faces denoted by the $F_i$ (in red) and one codimension $2$ face $\{0\}$ at the center (in green).
    \begin{figure}[t]
	    \centerline{\includegraphics[width=1.3\textwidth]{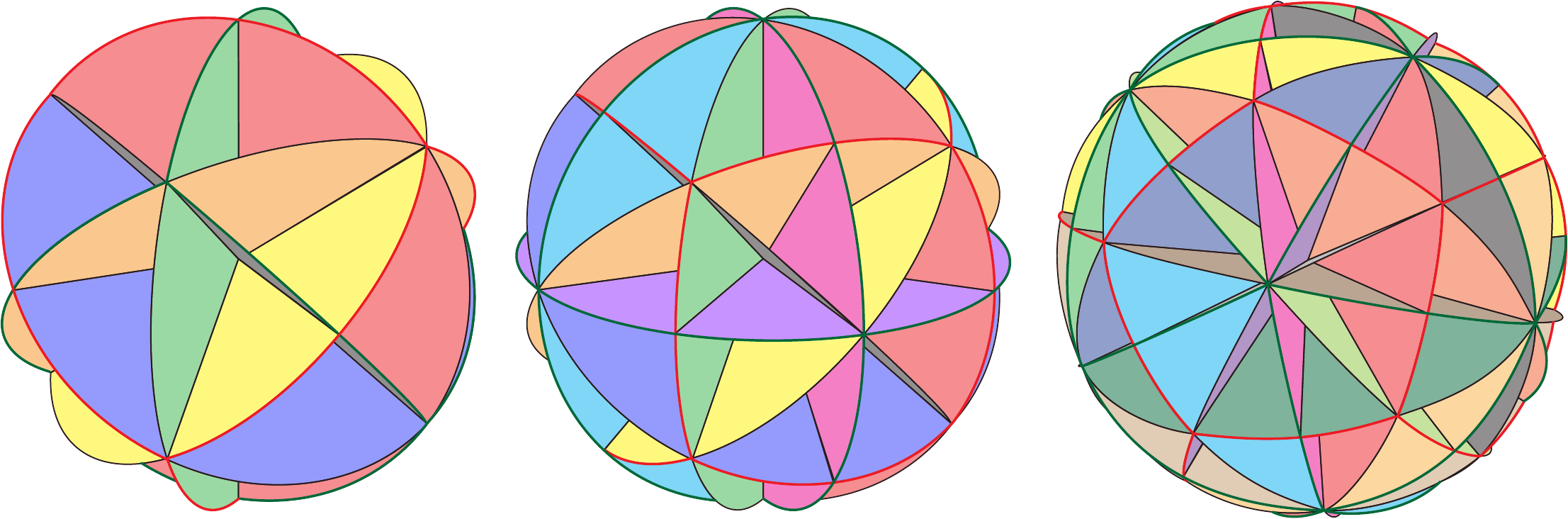}}
    	\caption{The type~$A_3$, $B_3$ and~$H_3$ Coxeter arrangements.}
	    \label{fig:CoxeterArrangements}
    \end{figure}
    \begin{figure}[b]
\DeclareDocumentCommand{\rs}{ O{1.1cm} O{->} m m O{0} O{0}} {
	\def \radius {#1}
	\def \inputPoints{#3}
	\def \excludeRoots{#4}
	\def \style {#2}
	\def \initialRotation {#5}
	\def \extraStyle {#6}

	\pgfmathtruncatemacro{\points}{\inputPoints * 2}
	\pgfmathsetmacro{\degrees}{360 / \points}
	
	\coordinate (0) at (0,0);
	
	\foreach \x in {1,...,\points}{%
		\pgfmathsetmacro{\location}{(\points+(\x-1))*\degrees + \initialRotation}
		
		\coordinate (\x) at (\location:\radius);
	}

	\ifthenelse{\equal{\excludeRoots}{}}{
		\foreach \x in {1,...,\points}{%
			\ifthenelse{\equal{\extraStyle}{0}}{
				\draw[\style] (0) -- (\x);
			}{
			\draw[\style, \extraStyle] (0) -- (\x);
		}
	}
}{
\foreach \x in {1,...,\points}{%
	\edef \showPoint {1};
	
	\foreach \y in \excludeRoots {
		\ifthenelse{\equal{\x}{\y}}{
			\xdef \showPoint {0};
		}{}
	}
	
	\ifthenelse{\equal{\showPoint}{1}}{
		\ifthenelse{\equal{\extraStyle}{0}}{
			\draw[\style] (0) -- (\x);
		}{
		\draw[\style, \extraStyle] (0) -- (\x);
	}
}{}
}
}  
}

\centerline{
    \begin{tikzpicture}
	    [
	    fdomain/.style={fill=blue!15!white,color=blue!15!white,opacity=0.3},
	    vdomain/.style={color=blue!85!white},
	    face/.style={draw=red!95!black,fill=red!95!black,color=red!95!black,ultra thick, -},
	    norm/.style={black,->},
	    vertex/.style={inner sep=1pt,circle,draw=green!85!white,fill=green!85!white,thick},
	    ]
	    %
	    \rs[1][]{3}{}[30]
	    %
	    \draw[norm] (0) -- (1) node[above right] {$-e_1$};
	    \draw[norm] (0) -- (2) node[above] {$-e_2$};
	    \draw[norm] (0) -- (3) node[above left] {$-e_3$};
	    \draw[norm] (0) -- (4) node[below left] {$e_1$};
	    \draw[norm] (0) -- (5) node[below] {$e_2$};
	    \draw[norm] (0) -- (6) node[below right] {$e_3$};
	    %
	    \rs[3.5][ultra thick]{3}{}[60]
	    \node[above right] at (1) {$H_{3}$};
	    \node[above left] at (2) {$H_{1}$};
	    \node[left] at (3) {$H_{2}$};
	    \rs[3][dashed]{3}{}[240]
	    \draw[face] (0) -- (2) node[right] {$\mathbf{F_0}$};
	    \draw[face] (0) -- (3) node[above] {$\mathbf{F_1}$};
	    \draw[face] (0) -- (4) node[right] {$\mathbf{F_2}$};
	    \draw[face] (0) -- (5) node[left] {$\mathbf{F_3}$};
	    \draw[face] (0) -- (6) node[above] {$\mathbf{F_4}$};
	    \draw[face] (0) -- (1) node[left] {$\mathbf{F_5}$};
	    %
	    %
	    \fill[fdomain] (0) -- (1) -- (2) -- cycle {};
	    \node[vdomain] at (270:2.2) {$B$};
	    \fill[fdomain] (0) -- (3) -- (2) -- cycle {};
	    \node[vdomain] at (330:2.2) {$R_1$};
	    \fill[fdomain] (0) -- (3) -- (4) -- cycle {};
	    \node[vdomain] at (30:2.2) {$R_2$};
	    \fill[fdomain] (0) -- (5) -- (4) -- cycle {};
	    \node[vdomain] at (90:2.2) {$R_3$};
	    \fill[fdomain] (0) -- (5) -- (6) -- cycle {};
	    \node[vdomain] at (150:2.2) {$R_4$};
	    \fill[fdomain] (0) -- (1) -- (6) -- cycle {};
	    \node[vdomain] at (210:2.2) {$R_5$};
	    \node[vertex] at (0) {};
    \end{tikzpicture}
}
        \caption{The type~$A_2$ Coxeter arrangement where $B$ is the intersection of the positive half-spaces of all hyperplanes. See \autoref{ex:Coxeter1}.}
	    \label{fig:A2}
    \end{figure}
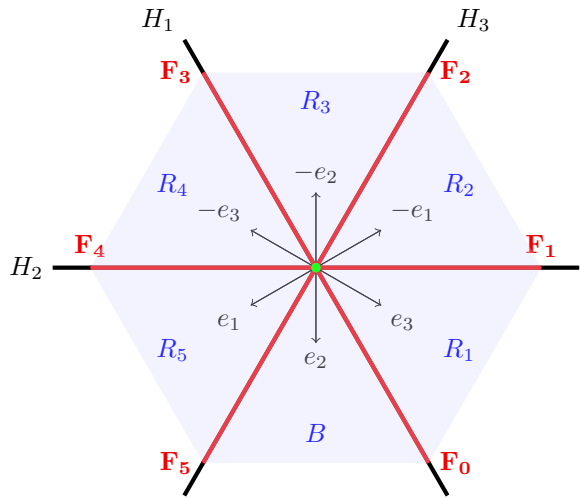
\end{example}


\subsection{Poset of regions}
\label{subsec:poset_of_regions}

Consider an arrangement $\arrangement$.
The \definition{separation set} of two regions $R, R' \in \regions$ is
\[
	S(R,R') \coloneqq \set{H \in \arrangement}{H \text{ separates } R \text{ from } R'}.
\]
We now choose $B$ to be a distinguished region of~$\arrangement$ called the \definition{base region}, and abbreviate~$S(B,R)$ into~$S(R)$.
The \definition{poset of regions} with respect to $B$ is the partial order~$\pr = (\regions, \prle)$ on the regions~$\regions$ of the arrangement~$\arrangement$ defined by
\[
	R \prle R' \quad \iff \quad S(R) \subseteq S(R').
\]
The poset of regions is graded by the cardinality  of the separation set $\abs{S(R)}$ of a region $R$.
The base region~$B$ is its minimum element and has rank ${\abs{S(B)} = \abs{\varnothing} = 0}$, and its opposite region~$-B$ is its maximum element and has rank~$\abs{S(-B)} = \abs{\arrangement}$.
Additionally, we have the following statement, see \eg \cite[Proposition 2.1]{Edelman}.

\begin{proposition}
    \label{prop:por_selfdual}
    The map ${R \mapsto -R \coloneqq \set{-v}{v \in R}}$ is a self-duality of the poset of regions $\pr$.
\end{proposition}

The reader is referred to \autoref{subsubsec:duality} for a definition of self-dual.
It is known that posets of regions associated to simplicial arrangements are lattices.

\begin{theorem}[{\cite[Theorems 3.1 and 3.4]{BjornerEdelmanZiegler}}]
	\label{thm:simplicial_implies_lattice}
    Suppose $\arrangement$ is essential.
	If the poset of regions $\pr$ is a lattice then the base region~$B$ is a simplicial region.
	Moreover, if $\arrangement$ is a simplicial arrangement then the poset of regions $\pr$ is a lattice for an arbitrary choice of base region~$B$.
\end{theorem}

\begin{example}
	\label{ex:Coxeter2}
    Following with \autoref{ex:Coxeter1}, the Hasse diagram of the poset of regions of a type $A_2$ Coxeter arrangement is given in \autoref{fig:A2_poset}.
    For Coxeter arrangements, the poset of regions is nothing more than the (right) weak order where the separation sets can be seen as inversion sets~\cite{Humphreys, BjornerBrenti}.
    In this example, we see that $R_1 \prle R_2$ since $S(R_1) = \{H_1\} \subseteq \{H_1, H_2\} = S(R_2)$, but $R_5 \not\prle R_2$ since $S(R_5) = \{H_3\} \not \subseteq \{H_1, H_2\} = S(R_2)$.
    The minimal element is~$B$, the maximal element is~$-B = R_3$.
\begin{figure}[b]
\DeclareDocumentCommand{\rs}{ O{1.1cm} O{->} m m O{0} O{0}} {
	\def \radius {#1}
	\def \inputPoints{#3}
	\def \excludeRoots{#4}
	\def \style {#2}
	\def \initialRotation {#5}
	\def \extraStyle {#6}

	\pgfmathtruncatemacro{\points}{\inputPoints * 2}
	\pgfmathsetmacro{\degrees}{360 / \points}
	
	\coordinate (0) at (0,0);
	
	\foreach \x in {1,...,\points}{%
		\pgfmathsetmacro{\location}{(\points+(\x-1))*\degrees + \initialRotation}
		
		\coordinate (\x) at (\location:\radius);
	}

	\ifthenelse{\equal{\excludeRoots}{}}{
		\foreach \x in {1,...,\points}{%
			\ifthenelse{\equal{\extraStyle}{0}}{
				\draw[\style] (0) -- (\x);
			}{
			\draw[\style, \extraStyle] (0) -- (\x);
		}
	}
}{
\foreach \x in {1,...,\points}{%
	\edef \showPoint {1};
	
	\foreach \y in \excludeRoots {
		\ifthenelse{\equal{\x}{\y}}{
			\xdef \showPoint {0};
		}{}
	}
	
	\ifthenelse{\equal{\showPoint}{1}}{
		\ifthenelse{\equal{\extraStyle}{0}}{
			\draw[\style] (0) -- (\x);
		}{
		\draw[\style, \extraStyle] (0) -- (\x);
	}
}{}
}
}  
}

\centerline{
    \begin{tikzpicture}
	    [
            vertex/.style={inner sep=1pt,circle,thick,fill=black,color=black},
            lline/.style={->, thick,shorten >=0.1cm,shorten <=0.1cm},
	    ]
	    %
	    \rs[3.5][dashed]{3}{}[60]
	    \node[above right] at (1) {$H_{3}$};
	    \node[above left] at (2) {$H_{1}$};
	    \node[left] at (3) {$H_{2}$};
	    %
	    %
	    \node[] at (270:2.5) {$B$};
	    \node[] at (330:2.5) {$R_1$};
	    \node[] at (30:2.5) {$R_2$};
	    \node[] at (90:2.5) {$R_3$};
	    \node[] at (150:2.5) {$R_4$};
	    \node[] at (210:2.5) {$R_5$};
	    \node[vertex] at (270:2) {};
	    \node[vertex] at (330:2) {};
	    \node[vertex] at (30:2) {};
	    \node[vertex] at (90:2) {};
	    \node[vertex] at (150:2) {};
	    \node[vertex] at (210:2) {};
        \draw[lline] (270:2) -- (330:2);
        \draw[lline] (270:2) -- (210:2);
        \draw[lline] (330:2) -- (30:2);
        \draw[lline] (210:2) -- (150:2);
        \draw[lline] (150:2) -- (90:2);
        \draw[lline] (30:2) -- (90:2);
    \end{tikzpicture}
}
   \caption{The lattice of regions associated to the type~$A_2$ Coxeter arrangement. See \autoref{ex:Coxeter2}.}
	\label{fig:A2_poset}
\end{figure}
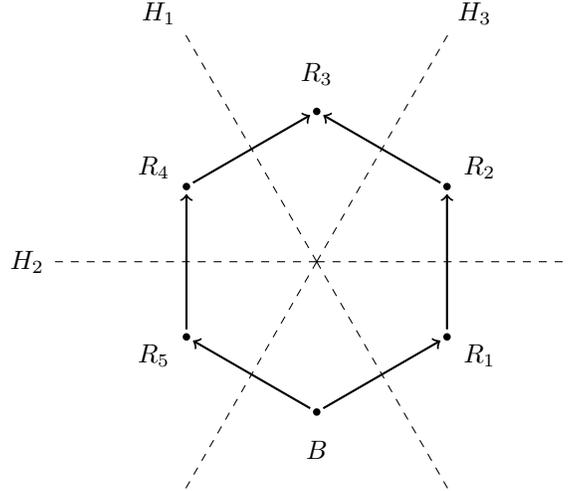
\end{example}


\subsection{Facial intervals}
\label{subsec:faces}

One of the interesting facts about the poset of regions is that it allows each face in $\faces$ to be described by a unique interval in $\pr$.
These intervals will be used to define the facial weak order.

\begin{proposition}
	\label{prop:facial_intervals}
	For any face~$F \in \faces$, the set~$\set{R \in \regions}{F \subseteq R}$ is an interval of the poset of regions~$\pr$.
	We denote it~$[m_F,M_F]$ and call it the \definition{facial interval} of~$F$.
	Moreover, $F = \bigcap_{R \in [m_F, M_F]} R$.
\end{proposition}

\begin{remark}
	A proof of the above \autoref{prop:facial_intervals} can be found in \cite[Lemma 4.2.12]{BjornerLasVergnasSturmfelsWhiteZiegler}. 
    It is based on the following geometric idea: the region~$m_F$ (resp.~$M_F$) is the region that is found when starting from any point in the relative interior of the face~$F$ and slightly moving in the direction of (resp.~away from) a point in the relative interior of the base region~$B$.
\end{remark}

For instance the interval corresponding to a region is the singleton constituted of that region.
Note that not all intervals of the poset of region~$\pr$ are facial intervals; only those of the form~$\set{R \in \regions}{F \subseteq R}$ for some face~$F \in \faces$.
Since~${F = \bigcap_{R \in [m_F, M_F]} R}$, we obtain the following corollary.

\begin{corollary}
    \label{cor:F_in_G_iff_faces_containment}
    For $F,G \in \faces$, we have $F \subseteq G \iff [m_F, M_F] \supseteq [m_G, M_G]$.
\end{corollary}

\begin{example}
	As we saw in \autoref{ex:Coxeter1}, there are $13$ faces in the arrangement of \autoref{fig:A2_poset}.
	For instance, the origin $\{0\}$ is represented by $[B, R_3]$ and $F_1$ is represented with $[R_1, R_2]$.
	Each region $R$ is given by the interval~$[R, R]$.
\end{example}


\subsection{Facial weak order}
\label{subsec:facial_weak_order}

We now state the definition of the facial weak order%
\footnote{
	Just like the poset of regions, it is tempting to call this order the poset of faces.
	However, the facial weak order IS NOT the classical face poset (the poset of faces ordered by inclusion).
	We have thus chosen to borrow the name facial weak order from the context of Coxeter groups studied in~\cite{DermenjianHohlwegPilaud} to the present context of hyperplane arrangements.
}
which will be the focus for the rest of the paper.

\begin{defn}
    \label{def:fwo_lower_upper_related}
	The \definition{facial weak order} is the order $\fwle$ on $\faces$ defined by
	\[
		F \fwle G \quad \iff \quad m_F \prle m_G \quad\text{and}\quad M_F \prle M_G
	\]
	where $[m_F, M_F]$ and $[m_G, M_G]$ are the facial intervals in $\pr$ associated to the faces $F$ and $G$ respectively.
    We denote by $\fw$ the poset~$(\faces, \fwle)$.
\end{defn}

\begin{example}
	\label{ex:Coxeter3}
	We give an example of the Hasse diagram of the facial weak order for the type $A_2$ Coxeter arrangement in \autoref{fig:A2FWOFaces}.
	As we saw in \autoref{ex:Coxeter1}, there are $13$ faces in the arrangement of \autoref{fig:A2_poset}, corresponding to the $13$ elements of the facial weak order.
	For example we have~$[B, R_5] \fwle [R_2, R_3]$ since~$B \prle R_2$ and~$R_5 \prle R_3$.
	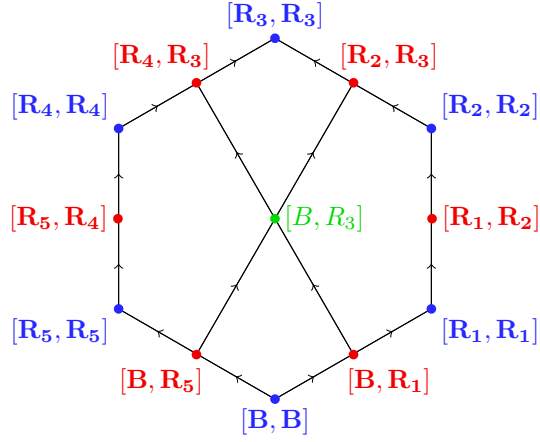
\begin{figure}[ht]
\DeclareDocumentCommand{\rs}{ O{1.1cm} O{->} m m O{0} O{0}} {
	\def \radius {#1}
	\def \inputPoints{#3}
	\def \excludeRoots{#4}
	\def \style {#2}
	\def \initialRotation {#5}
	\def \extraStyle {#6}

	\pgfmathtruncatemacro{\points}{\inputPoints * 2}
	\pgfmathsetmacro{\degrees}{360 / \points}
	
	\coordinate (0) at (0,0);
	
	\foreach \x in {1,...,\points}{%
		\pgfmathsetmacro{\location}{(\points+(\x-1))*\degrees + \initialRotation}
		
		\coordinate (\x) at (\location:\radius);
	}

	\ifthenelse{\equal{\excludeRoots}{}}{
		\foreach \x in {1,...,\points}{%
			\ifthenelse{\equal{\extraStyle}{0}}{
				\draw[\style] (0) -- (\x);
			}{
				\draw[\style, \extraStyle] (0) -- (\x);
			}
		}
	}{
		\foreach \x in {1,...,\points}{%
			\edef \showPoint {1};

			\foreach \y in \excludeRoots {
				\ifthenelse{\equal{\x}{\y}}{
					\xdef \showPoint {0};
				}{}
			}
			
			\ifthenelse{\equal{\showPoint}{1}}{
				\ifthenelse{\equal{\extraStyle}{0}}{
					\draw[\style] (0) -- (\x);
				}{
					\draw[\style, \extraStyle] (0) -- (\x);
				}
			}{}
		}
	}  
}

\centerline{
	\begin{tikzpicture}
		[
        scale=1.2,
        face/.style={color=red!95!black},
        fvertex/.style={inner sep=1pt,circle,draw=red!95!black,fill=red!95!black,thick},
        vertex/.style={inner sep=1pt,circle,draw=blue!85!white,fill=blue!85!white,thick},
        pvertex/.style={inner sep=1pt,circle,draw=green!85!black,fill=green!85!black,thick},
        chamber/.style={color=blue!85!white},
        poly/.style={color=green!85!black},
        fdomain/.style={fill=green!15!white,color=green!15!white,opacity=0.2},
		]
		\tikzset{->-/.style={decoration={
					markings,
					mark=at position #1 with {\arrow{>}}},postaction={decorate}}}
        \rs[2][white]{3}{}[30];
        \draw[->-=.25] (1) -- (2);
        \draw[->-=.75] (1) -- (2);
        \draw[->-=.25] (3) -- (2);
        \draw[->-=.75] (3) -- (2);
        \draw[->-=.25] (4) -- (3);
        \draw[->-=.75] (4) -- (3);
        \draw[->-=.25] (5) -- (4);
        \draw[->-=.75] (5) -- (4);
        \draw[->-=.25] (5) -- (6);
        \draw[->-=.75] (5) -- (6);
        \draw[->-=.25] (6) -- (1);
        \draw[->-=.75] (6) -- (1);
        %
        \node[fvertex] (f1) at (240:1.74) {};
        \node[fvertex] (f2) at (300:1.74) {};
        \node[fvertex] (f3) at (0:1.74) {};
        \node[fvertex] (f4) at (60:1.74) {};
        \node[fvertex] (f5) at (120:1.74) {};
        \node[fvertex] (f6) at (180:1.74) {};
        \draw[->-=.25] (f1) -- (f4);
        \draw[->-=.75] (f1) -- (f4);
        \draw[->-=.25] (f2) -- (f5);
        \draw[->-=.75] (f2) -- (f5);
        \node[vertex] at (1) {};
        \node[vertex] at (2) {};
        \node[vertex] at (3) {};
        \node[vertex] at (4) {};
        \node[vertex] at (5) {};
        \node[vertex] at (6) {};
        \node[chamber, above right] at (1) {$\mathbf{[R_2, R_2]}$};
        \node[chamber, above] at (2) {$\mathbf{[R_3,R_3]}$};
        \node[chamber, above left] at (3) {$\mathbf{[R_4,R_4]}$};
        \node[chamber, below left] at (4) {$\mathbf{[R_5,R_5]}$};
        \node[chamber, below] at (5) {$\mathbf{[{B},{B}]}$};
        \node[chamber, below right] at (6) {$\mathbf{[R_1,R_1]}$};
        \node[face] at (235:2.2) {$\mathbf{[{B},R_5]}$};
        \node[face] at (305:2.2) {$\mathbf{[{B},R_1]}$};
        \node[face] at (0:2.4) {$\mathbf{[R_1,R_2]}$};
        \node[face] at (55:2.2) {$\mathbf{[R_2,R_3]}$};
        \node[face] at (125:2.2) {$\mathbf{[R_4,R_3]}$};
        \node[face] at (180:2.4) {$\mathbf{[R_5,R_4]}$};
        \node[poly, right] at (0) {$[{B}, R_3]$};
        \node[pvertex] at (0) {};
	\end{tikzpicture}
}
	    \caption{The facial weak order labeled by facial intervals for the type $A_2$ Coxeter arrangement. See \autoref{ex:Coxeter3}.}
		\label{fig:A2FWOFaces}
	\end{figure}
\end{example}

The facial weak order was first defined for the braid arrangement by D.~Krob, M.~Latapy, J.-C.~Novelli, H.-D.~Phan and S.~Schwer in \cite{KrobLatapyNovelliPhanSchwer}.
It was then extended to arbitrary finite Coxeter arrangements by P.~Palacios and M.~Ronco in~\cite{PalaciosRonco}.
This order was studied in detail in \cite{DermenjianHohlwegPilaud}.
\autoref{def:fwo_lower_upper_related} was written there in Coxeter language.
Namely, for a Coxeter system $(W,S)$, the poset of regions is the right weak order~$\rwo$ on elements of~$W$.
The faces of the Coxeter arrangement correspond to the standard parabolic cosets $xW_I$ where $I \subseteq S$, ${W_I = \gen{I}}$, and~${x \in W^I = \set{w \in W}{\ell(w) \leq \ell(ws) \; \forall s \in I}}$.
In this case, the facial intervals are given by $[x,xw_{\circ, I}]$ where $w_{\circ, I}$ is the longest element in the parabolic subgroup~$W_I$.
The order $\fwle$ was given by $xW_I \fwle yW_J$ if and only if $x \rwo y$ and~${xw_{\circ, I} \rwo yw_{\circ, J}}$.

\begin{remark}
	The facial weak order~$\fw$ is clearly a poset (reflexive, antisymmetric and transitive) as the poset of regions is.
    In fact, the facial weak order~$\fw$ is the subposet induced by facial intervals in the poset of all intervals of the poset of regions~$\pr$ where $[a,b] < [c,d]$ if and only if $a \prle c$ and $b \prle d$.
\end{remark}

\begin{remark}
	\label{rem:subposet}
	Note that the poset of regions~$\pr$ is clearly the subposet of the facial weak order~$\fw$ induced by the singletons~$[R,R]$ for~$R \in \regions$.
	We will see in \autoref{prop:sublattice} that this observation also holds at the level of lattices when $\arrangement$ is simplicial.
\end{remark}


\subsection{Cover relations for the facial weak order}
\label{subsec:cover_relations}

For two faces $F$ and $G$ such that $F \fwle G$, recall that $m_F \prle m_G$ and $M_F \prle M_G$ by \autoref{def:fwo_lower_upper_related}.

\begin{proposition}
    \label{prop:two_implies_third}
    For any two faces~$F, G \in \faces$ then any two of the following conditions implies the third one:
    \begin{enumerate}
        \item \label{item:two_1} $F \subseteq G$ (resp. $F \supseteq G$),
        \item \label{item:two_2} $M_F = M_G$ (resp. $m_F = m_G$),
        \item \label{item:two_3} $F \fwle G$.
    \end{enumerate}
    
\end{proposition}

\begin{proof}
    Suppose first that $F \fwle G$ and $F \subseteq G$.
    By \autoref{cor:F_in_G_iff_faces_containment}, this implies the inclusion~${[m_F, M_F] \supseteq [m_G, M_G]}$ and therefore $M_G \prle M_F$.
    Furthermore, by \autoref{def:fwo_lower_upper_related} since ${F \fwle G}$ then $M_F \prle M_G$ forcing $M_F = M_G$ as desired.
    The proof that $F \fwle G$ and $F \supseteq G$ implies $m_F = m_G$ is similar.
	
    Suppose next that $F \fwle G$ and $M_F = M_G$.
    As $F \fwle G$ we have $m_F \prle m_G$ and therefore $[m_F, M_F] \supseteq [m_G, M_F] = [m_G, M_G]$.
    In other words $F \subseteq G$ by \autoref{cor:F_in_G_iff_faces_containment}.
    The proof that $F \fwle G$ and $m_F = m_G$ implies $F \supseteq G$ is similar.

    Finally, suppose $F\subseteq G$ and $M_F = M_G$.
    For $F \fwle G$ to hold, it suffices to show $m_F \prle m_G$.
    Since $F \subseteq G$, then by \autoref{cor:F_in_G_iff_faces_containment} $[m_F, M_F] \supseteq [m_G, M_G]$ and therefore $m_F \prle m_G$ as desired.
    The proof that $F \supseteq G$ and $m_F = m_G$ implies $F \fwle G$ is similar.
\end{proof}

The next proposition shows two types of cover relations for the facial weak order.
These will be shown to be precisely all cover relations in $\fw$  in  \autoref{thm:equivalence1}.

\begin{proposition}
    \label{prop:covers_for_fwo}
    For any two faces $F, G \in \faces$ such that $\abs{\dim F - \dim G} = 1$ and either
    \begin{itemize}
        \item $F \subseteq G$ and $M_F = M_G$, or
        \item $F \supseteq G$ and $m_F = m_G$
    \end{itemize}
    hold, then $F$ is covered by $G$, which we denote $F \fwcover G$.
\end{proposition}
\begin{proof}
    Assume that $F\subseteq G$ and $M_F = M_G$, the argument for the other case being symmetric.
    By \autoref{prop:two_implies_third}, $F \fwle G$.
    Let $X \in \faces \setm \{F, G\}$ be a face such that $F \fwl X \fwl G$.
    By definition of the facial weak order we have~${M_X = M_F= M_G}$.
    Then, by \autoref{prop:two_implies_third} again, $F \subseteq X \subseteq G$.
    Furthermore, since~${\abs{\dim F - \dim G} = 1}$, we necessarily have  $X = F$ or $X = G$.
    Hence $F$ is covered by $G$.
\end{proof}


\section{Geometric interpretations for the facial weak order}
\label{sec:Geometric_interpretations}

We describe in this section two different geometric interpretations for the facial weak order: first by the covectors of the corresponding oriented matroid, then by what we call root inversion sets which relates to the geometry of the corresponding zonotope.
We prove along the way that these various interpretations are equivalent.

Throughout this section, $\arrangement$ is a hyperplane arrangement.
We fix a normal vector~$e_H$ to each hyperplane~${H \in \arrangement}$, so that~$H = \set{v \in V}{\gen{e_H, v} = 0}$.
We consider the half spaces $H^+ = \set{v \in V}{\gen{e_H, v} \geq 0}$ and $H^- = \set{v \in V}{\gen{e_H, v} \leq 0}$ where the boundary in both cases is $H$.
For convenience, we choose the direction of the vector~$e_H$ such that the base region~$B$ lies in $H^+$.


\subsection{Covectors and oriented matroids}
\label{subsec:covectors}

In this section, we introduce basic oriented matroid terminology to deal geometrically with our hyperplane arrangements.
As we only consider hyperplane arrangements, we focus on realizable oriented matroids.
Moreover, we only consider covectors, and do not discuss other perspectives on oriented matroids.
A more general setting and background on oriented matroids can be found in the book by A.~Bj\"orner, M.~Vergas, B.~Sturmfels, N.~White and G.~M.~Ziegler \cite{BjornerLasVergnasSturmfelsWhiteZiegler}.

The \definition{sign map} of the hyperplane arrangement~$\arrangement$ is the map
\[
	\sigma: V  \to \{-, 0, +\}^{\arrangement}
\]
defined for $v \in V$ by $\sigma(v) = \big( \sigma_H(v) \big)_{H \in \arrangement}$ where 
\[
    \sigma_H(v) = \mathrm{sign}(\gen{v, e_H}) = \begin{cases}
        + & \text{if } \gen{v, e_H} > 0,\\
        - & \text{if } \gen{v, e_H} < 0,\\
        0 & \text{if } \gen{v, e_H} = 0.\\
    \end{cases}
\]
This map may be extended to assign to each face of $\arrangement$ a vector in $\{-, 0, +\}^{\arrangement}$ as follows.
Denote by $\int(F)$ the set of points in the relative interior of the face $F$.
The \definition{face sign map} of the hyperplane arrangement~$\arrangement$ is the map
\[
	\hat{\sigma}: \faces \to \{-, 0, +\}^{\arrangement}
\]
defined by~$\hat{\sigma}(F) = \sigma(v)$ for~$v \in \int(F)$.
This map is well-defined since, for arbitrary~$v$ and $w$ in $\int(F)$, we have that $\sigma(v) = \sigma(w)$.
However, note that for $v$ on the boundary of $F$ we could have that $\sigma_H(v) = 0$ even if $\hat{\sigma}_H(F) \ne 0$ for $H \in \arrangement$.
Thus for $v \in F \setm \int(F)$ either $\sigma_H(v) = \hat{\sigma}_H(F)$ or $\sigma_H(v) = 0$ for each $H \in \arrangement$.
Note that a face $F$ is easily recovered from its covector:
\[
    F = \bigcap_{H \in \arrangement} H^{\hat{\sigma}_H(F)}.
\]
By abuse of notation we let $\hat{\sigma}_H(F)$ be denoted by $F(H)$.
The sign vector~$\hat{\sigma}(F)$ is called \definition{covector} of the face~$F$, and the image $\covectors(\arrangement) \coloneqq \hat{\sigma}(\faces)$ of all faces in $\faces$ by the face sign map~$\hat{\sigma}$ is the \definition{set of covectors} of the hyperplane arrangement $\arrangement$. 

\begin{example}
	\label{ex:Coxeter4}
	We have represented in \autoref{fig:A2Covectors} the covectors of all faces of the type $A_2$ Coxeter arrangement in \autoref{fig:A2}.
	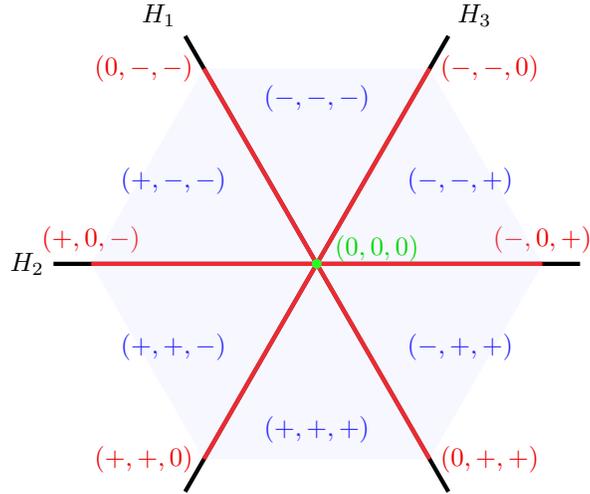
\begin{figure}[ht]
\DeclareDocumentCommand{\rs}{ O{1.1cm} O{->} m m O{0} O{0}} {
	\def \radius {#1}
	\def \inputPoints{#3}
	\def \excludeRoots{#4}
	\def \style {#2}
	\def \initialRotation {#5}
	\def \extraStyle {#6}

	\pgfmathtruncatemacro{\points}{\inputPoints * 2}
	\pgfmathsetmacro{\degrees}{360 / \points}
	
	\coordinate (0) at (0,0);
	
	\foreach \x in {1,...,\points}{%
		\pgfmathsetmacro{\location}{(\points+(\x-1))*\degrees + \initialRotation}
		
		\coordinate (\x) at (\location:\radius);
	}

	\ifthenelse{\equal{\excludeRoots}{}}{
		\foreach \x in {1,...,\points}{%
			\ifthenelse{\equal{\extraStyle}{0}}{
				\draw[\style] (0) -- (\x);
			}{
			\draw[\style, \extraStyle] (0) -- (\x);
		}
	}
}{
\foreach \x in {1,...,\points}{%
	\edef \showPoint {1};
	
	\foreach \y in \excludeRoots {
		\ifthenelse{\equal{\x}{\y}}{
			\xdef \showPoint {0};
		}{}
	}
	
	\ifthenelse{\equal{\showPoint}{1}}{
		\ifthenelse{\equal{\extraStyle}{0}}{
			\draw[\style] (0) -- (\x);
		}{
		\draw[\style, \extraStyle] (0) -- (\x);
	}
}{}
}
}  
}

\centerline{
    \begin{tikzpicture}
	    [
	    fdomain/.style={fill=blue!15!white,color=blue!15!white,opacity=0.2},
	    vdomain/.style={color=blue!85!white},
	    face/.style={draw=red!95!black,fill=red!95!black,color=red!95!black,ultra thick, -},
	    norm/.style={black,->},
	    vertex/.style={inner sep=1pt,circle,draw=green!85!white,fill=green!85!white,thick},
	    vvertex/.style={color=green!85!black},
	    ]
	    %
	    %
	    \rs[3.5][ultra thick]{3}{}[60]
	    \node[above right] at (1) {$H_{3}$};
	    \node[above left] at (2) {$H_{1}$};
	    \node[left] at (3) {$H_{2}$};
	    \rs[3][dashed]{3}{}[240]
	    \draw[face] (0) -- (1) node[left] {$(+,+,0)$};
	    \draw[face] (0) -- (2) node[right] {$(0,+,+)$};
	    \draw[face] (0) -- (3) node[above] {$(-,0,+)$};
	    \draw[face] (0) -- (4) node[right] {$(-,-,0)$};
	    \draw[face] (0) -- (5) node[left] {$(0,-,-)$};
	    \draw[face] (0) -- (6) node[above] {$(+,0,-)$};
	    %
	    %
	    \fill[fdomain] (0) -- (1) -- (2) -- cycle {};
	    \node[vdomain] at (270:2.2) {$(+,+,+)$};
	    \fill[fdomain] (0) -- (3) -- (2) -- cycle {};
	    \node[vdomain] at (330:2.2) {$(-,+,+)$};
	    \fill[fdomain] (0) -- (3) -- (4) -- cycle {};
	    \node[vdomain] at (30:2.2) {$(-,-,+)$};
	    \fill[fdomain] (0) -- (5) -- (4) -- cycle {};
	    \node[vdomain] at (90:2.2) {$(-,-,-)$};
	    \fill[fdomain] (0) -- (5) -- (6) -- cycle {};
	    \node[vdomain] at (150:2.2) {$(+,-,-)$};
	    \fill[fdomain] (0) -- (1) -- (6) -- cycle {};
	    \node[vdomain] at (210:2.2) {$(+,+,-)$};
	    \node[vertex] at (0) {};
	    \node[vvertex] at (0.8,0.2) {$(0,0,0)$};
    \end{tikzpicture}
}
		\caption{The type~$A_2$ Coxeter arrangement where the faces are identified by their associated covectors. See \autoref{ex:Coxeter4}.}
		\label{fig:A2Covectors}
	\end{figure}
\end{example}



We next define some useful operations on sign vectors which we use throughout this paper.
For two sign vectors $F, G \in \{ -, 0, +\}^\arrangement$, define
\begin{itemize}
\item the \definition{opposite} of $F$:
\(
    -F(H) = \begin{cases}
        + & \text{if } F(H) = -,\\
        - & \text{if } F(H) = +,\\
        0 & \text{if } F(H) = 0.
	\end{cases}
\)
\item the \definition{composition} of $F$ and $G$:
\(
    (F \circ G)(H) = \begin{cases}
        F(H) & \text{if } F(H) \ne 0,\\
        G(H) & \text{otherwise}.
	\end{cases}
\)
\item the \definition{reorientation} of a $F$ by $G$:
\(
    (F_{-G})(H) = \begin{cases}
        -F(H) & \text{if } G(H) = 0,\\
        F(H) & \text{otherwise}.
    \end{cases}
\)
\item the \definition{separation set}:
\(
    S(F, G) = \set{H \in \arrangement }{F(H) = -G(H) \ne 0}.
\)
\end{itemize}

\begin{example}
	\label{ex:Coxeter5}
	For instance, on the arrangement of \autoref{fig:A2Covectors}, for~$F = (-,0,+)$ and~$G = (0,-,-)$, we have
	\[
		-F = (+,0,-),
		\quad
		F \circ G = (-,-,+),
		\quad
		F_{-G} = (+,0,+)
		\quad\text{and}\quad
		S(F, G) = \{H_3\}.
	\]
\end{example}

\medskip
Note that if~$F$ and~$G$ are covectors in~$\covectors(\arrangement)$, then the opposite $-F$ of~$F$ and the composition $F \circ G$ of~$F$ and~$G$ are both covectors in $\covectors(\arrangement)$, or in other words faces in~$\faces$.
If furthermore, $G \subseteq F$ then the reorientation~$F_{-G}$ of $F$ by $G$ is a covector in~$\covectors(\arrangement)$ as well.
Moreover, $G \subseteq F_{-G}$ as faces.
Note that the separation set of regions from \autoref{subsec:poset_of_regions} is the same separation set as given for covectors of regions.
It is well-known that the set of covectors $\covectors(\arrangement)$ of the arrangement~$\arrangement$ is an oriented matroid in the sense of the following definition.
Cryptomorphic definitions for oriented matroids can be found in the book by A.~Bj\"orner, M.~Vergnas, B.~Sturmfels, N.~White and G.~M.~Ziegler \cite{BjornerLasVergnasSturmfelsWhiteZiegler}.

\begin{defn}
    \label{defn:oriented_matroid}
    An \definition{oriented matroid} is a pair $\big( \arrangement, \covectors \big)$ where~$\covectors$ is a collection of sign vectors in~$\{-, 0, +\}^\arrangement$ satisfying the following four properties:
    \begin{enumerate}
        \item \label{enum:covec_def_0} $\bf{0} \in \covectors$.
        \item \label{enum:covec_def_1} If $F \in \covectors$ then $(-F) \in \covectors$.
        \item \label{enum:covec_def_2} If $F, G \in \covectors$ then $(F \circ G) \in \covectors$.
        \item \label{enum:covec_def_3} \definition{Elimination axiom: } If $F, G \in \covectors$ and $H \in S(F,G)$ then there exists $X \in \covectors$ such that $X(H) = 0$ and $X(H') = (F \circ G)(H') = (G \circ F)(H')$ for all $H' \notin S(F, G)$.
    \end{enumerate}
\end{defn}

The notion of oriented matroids allows us to have a nice algebraic interpretation of what it means for a face $F$ to be a face of $G$.
We can either do a comparison between the two faces relative to the hyperplanes or we check how their covectors interact through composition.

\begin{proposition}
    \label{prop:face_inclusion_on_covectors}
	The following assertions are equivalent for two faces $F, G \in \faces$:
    \begin{enumerate}
        \item \label{item:inc1} $F \subseteq G$ as faces,
        \item \label{item:inc3} for all $H \in \arrangement$ either $F(H) = 0$ or $F(H) = G(H)$, and
        \item \label{item:inc2} $G = F \circ G$ as covectors.
    \end{enumerate}
\end{proposition}

\begin{proof}
    The equivalence of \autoref{item:inc2} and \autoref{item:inc3} is readily seen by definition of the composition~$F \circ G$.
    Furthermore, if $F \subseteq G$ as faces then it is readily seen that for all $H \in \arrangement$ either $F(H) = 0$ or $F(H) = G(H)$.

    It remains to show that \autoref{item:inc3} implies \autoref{item:inc1}.
    Suppose contrarily that $F \not\subseteq G$.
    If~${F \supsetneq G}$ then there exists some $H \in \arrangement$ such that $G(H) = 0 \ne F(H)$ since $F \ne G$.
    Else if~${F \not\supseteq G}$ and $F \not\subseteq G$ there is an $H' \in \arrangement$ which separates $G$ and $F$.
    In other words~${0 \ne G(H') = -F(H')}$.
\end{proof}

In fact, the sign of a face relative to a hyperplane tells us a lot about the regions containing the face.
Recall that for a region $R$, the separation set between the base region $B$ and $R$ is denoted by $S(R)$.

\begin{lemma}
    \label{lem:covector_and_sepSet}
    For a face~$F \in \faces$ with facial interval $[m_F, M_F]$ and a hyperplane~$H \in \arrangement$,
    \begin{enumerate}
        \item $F(H) = -$ if and only if $H \in S(m_F)$,
        \item $F(H) = 0$ if and only if $H \in S(M_F)$ and $H \notin S(m_F)$,
        \item $F(H) = +$ if and only if $H \notin S(M_F)$.
    \end{enumerate}
    In other words, 
    \[
    	S(m_F) = \set{H \in \arrangement}{F(H) < 0}
		\quad\text{and}\quad
		S(M_F) = \set{H \in \arrangement}{F(H) \leq 0}.
    \]
\end{lemma}

\begin{proof}
    We show the first case, the other two being similar.
    First, recall that by the definition of interval, $H \in S(m_F)$ if and only if for all $R \in [m_F, M_F]$ then~${H \in S(R)}$.
    In other words, if and only if $H$ separates the base region $B$ from $F$.
    This is true if and only if for some $v \in \int(F)$ then $\gen{v, e_H} < 0$ since $B \subseteq H^+$ by our chosen orientation given at the beginning of this section.
    In other words, if and only if~${F(H) = -}$.
\end{proof}

This lemma allows us to be a little more precise as to which faces are faces of $B$ and gives us a stronger method of finding these faces.
Not only are the faces of $B$ the faces with all non-negative components in their covector, but, in fact, we can strengthen this by only needing to look at the hyperplanes which bound $B$.

\begin{corollary}
    \label{cor:containment_B}
    The following assertions are equivalent for a face $F \in \faces$:
    \begin{enumerate}
        \item \label{enum:base_subset_1} $F \subseteq B$,
        \item \label{enum:base_subset_2} $F(H) \geq 0$ for all $H$ bounding $B$, and
        \item \label{enum:base_subset_3} $F(H) \geq 0$ for all $H \in \arrangement$.
    \end{enumerate}
\end{corollary}

\begin{proof}
    The points \autoref{enum:base_subset_1} and \autoref{enum:base_subset_3} are equivalent by \autoref{prop:face_inclusion_on_covectors} and the fact that~$B(H) > 0$ for all~$H \in \arrangement$.
    Additionally, \autoref{enum:base_subset_3} implies \autoref{enum:base_subset_2} is readily seen.

    To show that \autoref{enum:base_subset_2} implies \autoref{enum:base_subset_3}, let $\boundary$ be the set of boundary hyperplanes of $B$.
    Suppose that there exists $H \in \arrangement \setm \boundary$ such that $F(H) = -$.
    Then $H \in S(m_F)$ by \autoref{lem:covector_and_sepSet}, therefore $m_F \ne B$.
    This implies $H' \in S(m_F)$ for some $H' \in \boundary$ as well, since some $H' \in \boundary$ must separate $B$ and $m_F$ by definition of $\boundary$.
    In other words,~${F(H') = -}$ for some $H' \in \boundary$.
\end{proof}

We conclude with an observation which ensures that a face is not contained in a given hyperplane.

\begin{lemma}
    \label{lem:Y_(_X_gives_nonzero_X*(H)}
    Let $F$ and $G$ be two distinct faces in $\faces$.
    If there exists $H \in \arrangement$ such that~${G = F \cap H}$, then $F(H) \ne 0$.
\end{lemma}

\begin{proof}
    Suppose contrarily that $F(H) = 0$.
    Then $F = F \cap H = G$ contradicting that~$F$ and~$G$ are distinct.
\end{proof}


\subsection{Covectors and the facial weak order}

It is well-known that the face poset of the arrangement~$\arrangement$ can be interpreted as the poset of covectors of~$\covectors(\arrangement)$ ordered coordinatewise by~$0 < -$ and~$0 < +$.
Adding a maximum element to both posets allows us to interpret the face lattice as a lattice of covectors.
Here, we consider instead a twisted order that relates to the facial weak order.

\begin{defn}
    \label{def:leqL}
    Given two covectors $F, G \in \covectors(\arrangement)$, let the order $\cole$ be defined by
    \[
	    F \cole G \quad \iff \quad G(H) \leq F(H) \text{ for all } H \in \arrangement,
    \]
    where the order on signs is the natural order~$- < 0 < +$.
\end{defn}

We are ready to state our first main theorem, stating the equivalence between three descriptions of the facial weak order using \autoref{def:fwo_lower_upper_related}, \autoref{prop:covers_for_fwo} and \autoref{def:leqL} respectively. 

\begin{thm}
	\label{thm:equivalence1}
	The following assertions are equivalent for two faces $F, G \in \faces$:
	\begin{enumerate}
		\item \label{item:fwo_faces} $F \fwle G$ in the facial weak order $\fw$,
        \item \label{item:fwo_covers} there exists a sequence of faces $F = F_1, F_2, \ldots, F_n = G$ such that for each $i$, $\abs{\dim F_i - \dim F_{i+1}} = 1$ and either $F_i \subseteq F_{i+1}$ and $M_{F_i} = M_{F_{i+1}}$ or $F_{i+1} \subseteq F_i$ and $m_{F_i} = m_{F_{i+1}}$.
        \item \label{item:fwo_covects} $F \cole G$ in terms of covectors.
	\end{enumerate}
\end{thm}

Due to this theorem, the two classes of cover relations from \autoref{prop:covers_for_fwo} describe all the cover relations for the facial weak order.

\begin{corollary}
    \label{cor:fwo_covers}
    For two faces~$F, G \in \faces$, we have $F \fwcover G$ in the facial weak order if and only if $\abs{\dim F - \dim G} = 1$ and either~${F \subseteq G}$ and $M_F = M_G$ or $G \subseteq F$ and $m_F = m_G$ if and only if $F \fwle G$, $\abs{\dim F - \dim G} = 1$ and either $F \subseteq G$ or~${G\subseteq F}$.
\end{corollary}

Before proving \autoref{thm:equivalence1}, we need the following two lemmas. 

\begin{lemma}
	\label{lem:F_leq_G_=>_F_leq_FcG_leq_GcF_leq_G}
	For $F, G \in \covectors(\arrangement)$, if $F \cole G$ then $F \cole F \circ G \cole G \circ F \cole G$.
\end{lemma}

\begin{proof}
    Suppose $F \cole G$, \ie for all $H \in \arrangement$, we have~$G(H) \leq F(H)$. Then
    \begin{itemize}
    	\item if $G(H) = +$ then $G(H) = (G \circ F)(H) = (F \circ G)(H) = F(H) = +$,
    	\item if $G(H) = 0$ then $G(H) \leq (G \circ F)(H) = (F \circ G)(H) = F(H)$,
    	\item if $G(H) = -$ then $G(H) = (G \circ F)(H) = - \leq (F \circ G)(H) \leq F(H)$.
    	The first inequality is an equality when $F(H) = 0$ else the second inequality becomes an equality.
	\end{itemize}
    Therefore in all three cases we have $G(H) \leq (G \circ F)(H) \leq (F \circ G)(H) \leq F(H)$ for arbitrary $H$ giving us the desired result.
\end{proof}

\begin{lemma}
    \label{lem:middle_covector_exists}
    If $F \col G$ and $S(F,G) \ne \varnothing$, then there exists ${F \col X \col G}$.
\end{lemma}

\begin{proof}
    Since $S(F,G)$ is non-empty, by the elimination axiom in \autoref{defn:oriented_matroid} for each $H \in S(F,G)$ there exists a $X \in \covectors(\arrangement)$ such that $X(H) = 0$ and for all~${H'\notin S(F,G)}$ then $X(H') = (F \circ G)(H') = (G \circ F)(H')$.
    Thus let $H$ be an arbitrary hyperplane in~$S(F,G)$ and let~$X$ be the associated covector in~${\covectors(\arrangement)}$.
    
    Since $H \in S(F,G)$ and $G(H) \leq F(H)$ we are forced to have~${G(H) = -}$ and~${F(H) = +}$.
    Furthermore, since $X(H) = 0$ we see that our three faces are distinct,~${F \ne X \ne G}$.
    It therefore suffices to show that $G(H') \leq X(H') \leq F(H')$ for all~${H' \in \arrangement \setm \{H\}}$.

    Suppose first that $H' \in S(F,G)$.
    Since $G(H') \leq F(H')$ and ${F(H') = -G(H') \ne 0}$ then $G(H') = -$ and $F(H') = +$.
    Thus $G(H') \leq X(H') \leq F(H')$ as desired.

    Suppose next that $H' \notin S(F,G)$.
    Since $G(H') \leq F(H')$ and ${-G(H') \ne F(H')}$ or ${G(H') = F(H') = 0}$, there are three cases to consider:
    \begin{itemize}
	    \item if $F(H') = G(H')$ then $F(H') = G(H') = (F \circ G)(H') = X(H')$,
    	\item if $F(H') = 0$ and $G(H') = -$ then $X(H') = (F \circ G)(H') = G(H') < F(H)$,
		\item if $F(H') = +$ and $G(H') = 0$ then $X(H') = (F \circ G)(H') = F(H') > G(H)$.
	\end{itemize}
    Therefore $G(H') \leq X(H') \leq F(H')$ and thus $F \col X \col G$.
\end{proof}

We now prove \autoref{thm:equivalence1}.

\begin{proof}[Proof of \autoref{thm:equivalence1}]
    We show that the points \autoref{item:fwo_faces}, \autoref{item:fwo_covers}, and \autoref{item:fwo_covects} are equivalent by showing the implications \autoref{item:fwo_covers} $\Rightarrow$ \autoref{item:fwo_faces} $\Rightarrow$ \autoref{item:fwo_covects} $\Rightarrow$ \autoref{item:fwo_covers}.

	\medskip\noindent
    \boxed{\textbf{\autoref{item:fwo_covers} $\Rightarrow$ \autoref{item:fwo_faces}}}
    By \autoref{prop:covers_for_fwo} the sequence~${F_1, \ldots, F_n}$ gives a a chain of covers~${F = F_1 \fwcover F_2 \fwcover \cdots \fwcover F_n = G}$ and therefore $F \fwle G$ as desired.
	
	\medskip\noindent
    \boxed{\textbf{\autoref{item:fwo_faces} $\Rightarrow$ \autoref{item:fwo_covects}}}
    Suppose $F \fwle G$ in the facial weak order, \ie $m_F \prle m_G$ and~${M_F \prle M_G}$.
    To show $F \cole G$ it suffices to show $G(H) \leq F(H)$ for arbitrary hyperplane~$H \in \arrangement$.
    If $G(H) = -$, then $G(H) \leq F(H)$ always.
    If $G(H) = +$ then by \autoref{lem:covector_and_sepSet}, $H \notin S(M_G)$.
    But since $M_F \prle M_G$ then $S(M_F) \subseteq S(M_G)$, in other words, $H \notin S(M_F)$.
    Applying \autoref{lem:covector_and_sepSet} again gives $F(H) = +$.
    Finally, if $G(H) = 0$ then by \autoref{lem:covector_and_sepSet}, $H \in S(M_G) \setm S(m_G)$.
    Therefore $H \notin S(m_G)$ and since $m_F \prle m_G$ we get $H \notin S(m_F)$.
    Thus by \autoref{lem:covector_and_sepSet}, $F(H) \ne -$ and~${G(H) = 0 \leq F(H)}$ as desired.
	
	\medskip\noindent
    \boxed{\textbf{\autoref{item:fwo_covects} $\Rightarrow$ \autoref{item:fwo_covers}}}
    We do this by induction on the path length from $F$ to $G$.
    Our base case of $F = G$ trivially holds.
	Suppose now that $F \col G$.
    By \autoref{lem:F_leq_G_=>_F_leq_FcG_leq_GcF_leq_G}, we have $F \cole F \circ G \cole G$.
    There are three cases two consider: 
    \begin{itemize}
    \item Suppose first that our inequalities are strict, \ie $F \col F \circ G \col G$.
	Then by induction $F \col F \circ G$ and $F \circ G \col G$ gives a chain of covers $\fwcover$ such that $F = F_1 \fwcover \cdots \fwcover F_i \fwcover F\circ G \fwcover G_1 \fwcover \cdots \fwcover G_j = G$.
	
    \item If $G = F \circ G$, then by  \autoref{prop:face_inclusion_on_covectors}, $F \subseteq G$.
        In particular, there exists a chain of faces, $F = F_0 \subseteq F_1 \subseteq \cdots \subseteq F_n = G$ such that~${\abs{\dim F_i - \dim F_{i - 1}} = 1}$ for all $i$ by the face lattice being graded.
        It remains to show~${M_{F_i} = M_{F_{i+1}}}$.
        Since $F_i \subseteq F_{i+1}$ then for each $H \in \arrangement$, either $F_i(H) = 0$ or~${F_i(H) = F_{i+1}(H)}$.
        If $F_i(H) = F_{i+1}(H)$ then $H \in S(M_F)$ if and only if~${H \in S(M_{F_{i+1}})}$.
        If~${F_i(H) = 0}$ then $F_{i+1}(H) = -$ (else $F(H) = 0$ and $G(H) = +$ by inclusion, contradicting the fact that $F \col G$).
        By \autoref{lem:covector_and_sepSet}~${F_i(H) = 0}$ implies~${H \in S(M_{F_i})}$ and $F_{i+1}(H) = -$ implies~${H \in S(M_{F_{i+1}})}$.
        Therefore~${S(M_{F_i}) = S(M_{F_{i+1}})}$ implying $M_{F_i} = M_{F_{i+1}}$ as desired.
	
    \item If $F = F \circ G$ we have two further cases to consider.
    First, if 
    \[
        S(F, G) = \set{H \in \arrangement}{F(H) = -G(H) \ne 0} = \varnothing
    \]
    then $G \circ F = F \circ G = F$.
    In particular, by  \autoref{prop:face_inclusion_on_covectors}, as $F = G \circ F$ then $G \subseteq F$ and using the sequences of faces $G = F_n, F_{n-1}, \ldots, F_1 = F$, then as in the previous case we have $\abs{\dim F_i - \dim F_{i+1}}$, $F_i \supseteq F_{i+1}$ and~${m_{F_{i}} = m_{F_{i+1}}}$ as desired.
    Finally, suppose $S(F,G) \ne \varnothing$.
    By \autoref{lem:middle_covector_exists} there exists a $X$ such that $F \col X \col G$.
    Thus, by inducting on this gives us the desired result.
    \qedhere
    \end{itemize}
\end{proof}

Moreover, using the covector definition, we show that the structure of an interval in $\fw$ is not altered by a change of base region as long as the new region is below the bottom element of our interval.

\begin{proposition}
	\label{lem_rotation}
    Let $X,Y$ be covectors in $\covectors(\arrangement)$ such that $X \fwle Y$ in $\fw$.
    If $B'$ is a region such that $B' \fwle X$ in $\fw$, then the intervals $[X,Y]$ in~$\fw$ and in $\fw[\arrangement][B']$ are isomorphic.
\end{proposition}

\begin{proof}
    Changing the base region from $B$ to $B'$ switches the orientation on any hyperplane in the separation set $S(B,B')$ and leaves the other hyperplanes with the same orientation.
    Since $B' \fwle X$, we have $X(H)=-$ whenever $H\in S(B,B')$ where $B$ is the base region.
    Hence, $Z(H)=-$ as well whenever $X \fwle Z$.
    After the reorientation, $X(H) = + = Z(H)$ for $H\in S(B,B')$.
    As the orientations of the hyperplanes not in $S(B,B')$ are unchanged, we conclude that the interval $[X,Y]$ is the same in $\fw[\arrangement][B']$ as in $\fw$.
\end{proof}

We finally derive a criterion to compare two faces of the base region~$B$ in the facial weak order.

\begin{corollary}
	\label{cor:fwle_B}
	For any faces~$F,G$ of the base region~$B$, we have~$F \supseteq G$ if and only if $F \fwle G$.
	Similarly, for any faces~$F,G$ of the region~$-B$ opposite to the base region~$B$, we have~$F \subseteq G$ if and only if $F \fwle G$.
\end{corollary}
\begin{proof}
	Consider a hyperplane~$H \in \arrangement$.
	Since~$F$ is a face of the base region~$B$, we have~$F(H) \ge 0$ by \autoref{cor:containment_B}.
    Since~$F \supseteq G$, we have~${G(H) = 0}$ or~${G(H) = F(H)}$ by \autoref{prop:face_inclusion_on_covectors}.
	Therefore, $F(H) \ge G(H)$ in both cases.
    We conclude that~${F \fwle G}$.
    The converse can be deduced from \autoref{prop:face_inclusion_on_covectors}.
	The proof for the second assertion is identical.
\end{proof}


\subsection{Root inversion sets}
\label{subsec:root_inversion_sets}

We now provide an alternative combinatorial encoding of the covectors in terms of certain sets of normal vectors that will be related to the geometry of the corresponding zonotope in the next section.
Recall that, by convention in this paper, $e_H$ is the fixed normal vector to the hyperplane $H \in \arrangement$ such that the base region $B$ lies in $H^+$.
We need the following three sets:
\[
\Phi_\arrangement^+ \coloneqq \set{e_H}{H \in \arrangement},
\quad
\Phi_\arrangement^- \coloneqq \set{-e_H}{H \in \arrangement},
\quad\text{ and }\quad
\Phi_\arrangement \coloneqq \Phi_\arrangement^+ \cup \Phi_\arrangement^-.
\]
We call the elements in $\Phi_\arrangement$ the \definition{roots}%
\footnote{This terminology is once again inherited from Coxeter systems, but it should be noted that these roots do not necessarily form root systems.}
of the arrangement $\arrangement$ and the elements in~$\Phi_\arrangement^+$ and~$\Phi_\arrangement^-$ the positive and negative roots respectively.
For~$X \subseteq \Phi_\arrangement$, we denote by~$X^+ \coloneqq X \cap \Phi_\arrangement^+$ the positive part and by~$X^- \coloneqq X \cap \Phi_\arrangement^-$ the negative part.
An example of this construction is given in \autoref{fig:A2} where the roots give the root system for the type $A_2$ Coxeter arrangement.

\begin{defn}
The \definition{root inversion set} of a face $F \in \faces$ is
\[
    \rootInversionSet(F) = \set{e \in \Phi_\arrangement}{\gen{x, e} \leq 0 , \text{ for some } x \in \int(F)}.
\]
\end{defn}

\begin{figure}[t]
\DeclareDocumentCommand{\rs}{ O{1.1cm} O{->} m m O{0} O{0}} {
	\def \radius {#1}
	\def \inputPoints{#3}
	\def \excludeRoots{#4}
	\def \style {#2}
	\def \initialRotation {#5}
	\def \extraStyle {#6}

	\pgfmathtruncatemacro{\points}{\inputPoints * 2}
	\pgfmathsetmacro{\degrees}{360 / \points}
	
	\coordinate (0) at (0,0);
	
	\foreach \x in {1,...,\points}{%
		\pgfmathsetmacro{\location}{(\points+(\x-1))*\degrees + \initialRotation}
		
		\coordinate (\x) at (\location:\radius);
	}

	\ifthenelse{\equal{\excludeRoots}{}}{
		\foreach \x in {1,...,\points}{%
			\ifthenelse{\equal{\extraStyle}{0}}{
				\draw[\style] (0) -- (\x);
			}{
				\draw[\style, \extraStyle] (0) -- (\x);
			}
		}
	}{
		\foreach \x in {1,...,\points}{%
			\edef \showPoint {1};

			\foreach \y in \excludeRoots {
				\ifthenelse{\equal{\x}{\y}}{
					\xdef \showPoint {0};
				}{}
			}
			
			\ifthenelse{\equal{\showPoint}{1}}{
				\ifthenelse{\equal{\extraStyle}{0}}{
					\draw[\style] (0) -- (\x);
				}{
					\draw[\style, \extraStyle] (0) -- (\x);
				}
			}{}
		}
	}  
}

\centerline{
	\begin{tikzpicture}
		[
            scale=1.2,
        vface/.style={color=red!95!black},
        face/.style={draw=red!95!black,fill=red!95!black,color=red!95!black},
        vertex/.style={inner sep=1pt,circle,draw=blue!85!white,fill=blue!85!white,thick},
        vvertex/.style={color=blue!85!white},
        vdomain/.style={color=green!85!black},
        fdomain/.style={fill=green!15!white,color=green!15!white,opacity=0.2},
		]
        \rs[2][white]{3}{}[30];
        \draw[face] (1) -- (2);
        \draw[face] (2) -- (3);
        \draw[face] (3) -- (4);
        \draw[face] (4) -- (5);
        \draw[face] (5) -- (6);
        \draw[face] (6) -- (1);
        \node[vertex] at (1) {};
        \node[vertex] at (2) {};
        \node[vertex] at (3) {};
        \node[vertex] at (4) {};
        \node[vertex] at (5) {};
        \node[vertex] at (6) {};
        \node[vvertex, above right] at (1) {$\mathbf{\tau(R_2)}$};
        \node[vvertex, above] at (2) {$\mathbf{\tau(R_3)}$};
        \node[vvertex, above left] at (3) {$\mathbf{\tau(R_4)}$};
        \node[vvertex, below left] at (4) {$\mathbf{\tau(R_5)}$};
        \node[vvertex, below] at (5) {$\mathbf{\tau({B})}$};
        \node[vvertex, below right] at (6) {$\mathbf{\tau(R_1)}$};
        \node[vface] at (240:2.2) {$\mathbf{\tau(F_5)}$};
        \node[vface] at (300:2.2) {$\mathbf{\tau(F_0)}$};
        \node[vface] at (0:2.2) {$\mathbf{\tau(F_1)}$};
        \node[vface] at (60:2.2) {$\mathbf{\tau(F_2)}$};
        \node[vface] at (120:2.2) {$\mathbf{\tau(F_3)}$};
        \node[vface] at (180:2.2) {$\mathbf{\tau(F_4)}$};
        \fill[fdomain] (6) -- (1) -- (2) -- (3) -- (4) -- (5) -- cycle {};
        \node[vdomain] at (0) {$\left\{ 0 \right\}$};
	\end{tikzpicture}
	\begin{tikzpicture}
        [scale=1.2,
        vface/.style={color=red!95!black},
        face/.style={draw=red!95!black,fill=red!95!black,color=red!95!black},
        vertex/.style={inner sep=1pt,circle,draw=blue!85!white,fill=blue!85!white,thick},
        vvertex/.style={color=blue!85!white},
        vdomain/.style={color=green!85!black},
        fdomain/.style={fill=green!15!white,color=green!15!white,opacity=0.2},
		]
		%
        \rs[2][white]{3}{}[30];
        \fill[fdomain] (6) -- (1) -- (2) -- (3) -- (4) -- (5) -- cycle {};
        \draw[face] (1) -- (2);
        \draw[face] (2) -- (3);
        \draw[face] (3) -- (4);
        \draw[face] (4) -- (5);
        \draw[face] (5) -- (6);
        \draw[face] (6) -- (1);
        \coordinate (c0) at (270:2);
        \coordinate (c1) at (330:2);
        \coordinate (c2) at (30:2);
        \coordinate (c3) at (90:2);
        \coordinate (c4) at (150:2);
        \coordinate (c5) at (210:2);
        \coordinate (f5) at (240:1.732);
        \coordinate (f0) at (300:1.732);
        \coordinate (f1) at (0:1.732);
        \coordinate (f2) at (60:1.732);
        \coordinate (f3) at (120:1.732);
        \coordinate (f4) at (180:1.732);
		%
		\begin{scope}[shift={(c4)}, scale=0.5,vvertex]
			\node[above left] {$\rootInversionSet(R_4)$};
			\begin{scope}[scale=0.8]
				\rs[1][ultra thick]{3}{1,2,3}[90][->]
			\end{scope}
		\end{scope}
		\begin{scope}[shift={(c3)}, scale=0.5,vvertex]
			\node[above] {$\rootInversionSet(R_3)$};
			\begin{scope}[scale=0.8]
				\rs[1][ultra thick]{3}{1,2,6}[90][->]
			\end{scope}
		\end{scope}
		\begin{scope}[shift={(c5)}, scale=0.5,vvertex]
			\node[below left] {$\rootInversionSet(R_5)$};
			\begin{scope}[scale=0.8]
				\rs[1][ultra thick]{3}{2,3,4}[90][->]
			\end{scope}
		\end{scope}
		\begin{scope}[shift={(c2)}, scale=0.5,vvertex]
			\node[above right] {$\rootInversionSet(R_2)$};
			\begin{scope}[scale=0.8]
				\rs[1][ultra thick]{3}{1,5,6}[90][->]
			\end{scope}
		\end{scope}
		\begin{scope}[shift={(c0)}, scale=0.5,vvertex]
            \node[below] {$\rootInversionSet({B})$};
			\begin{scope}[scale=0.8]
				\rs[1][ultra thick]{3}{3,4,5}[90][->]
			\end{scope}
		\end{scope}
		\begin{scope}[shift={(c1)}, scale=0.5,vvertex]
			\node[below right] {$\rootInversionSet(R_1)$};
			\begin{scope}[scale=0.8]
				\rs[1][ultra thick]{3}{4,5,6}[90][->]
			\end{scope}
		\end{scope}
		%
		\begin{scope}[shift={(f4)}, scale=0.5, face]
			\node[vface,left] {$\rootInversionSet(F_4)$};
			\begin{scope}[scale=0.8]
				\rs[1][ultra thick]{3}{2,3}[90][->]
			\end{scope}
		\end{scope}
		\begin{scope}[shift={(f3)}, scale=0.5,face]
			\node[vface,above left] {$\rootInversionSet(F_3)$};
			\begin{scope}[scale=0.8]
				\rs[1][ultra thick]{3}{1,2}[90][->]
			\end{scope}
		\end{scope}
		\begin{scope}[shift={(f5)}, scale=0.5,face]
            \node[vface,below left] {$\rootInversionSet(F_5)$};
			\begin{scope}[scale=0.8]
				\rs[1][ultra thick]{3}{3,4}[90][->]
			\end{scope}
		\end{scope}
		\begin{scope}[shift={(f2)}, scale=0.5,face]
			\node[vface,above right] {$\rootInversionSet(F_2)$};
			\begin{scope}[scale=0.8]
				\rs[1][ultra thick]{3}{1,6}[90][->]
			\end{scope}
		\end{scope}
		\begin{scope}[shift={(f0)}, scale=0.5,face]
			\node[vface,below right] {$\rootInversionSet(F_0)$};
			\begin{scope}[scale=0.8]
				\rs[1][ultra thick]{3}{4,5}[90][->]
			\end{scope}
		\end{scope}
		\begin{scope}[shift={(f1)}, scale=0.5,face]
            \node[vface,right] {$\rootInversionSet(F_1)$};
			\begin{scope}[scale=0.8]
				\rs[1][ultra thick]{3}{5,6}[90][->]
			\end{scope}
		\end{scope}
		%
        \begin{scope}[shift={(0,0)}, scale=0.5,vdomain]
            \node[below,vdomain] at (0, -.7) {$\rootInversionSet(\left\{ 0 \right\})$};
			\begin{scope}[scale=0.8]
				\rs[1][ultra thick]{3}{}[90][->]
			\end{scope}
		\end{scope}
	\end{tikzpicture}
}
	\caption{The type~$A_2$ Coxeter arrangement. On the left is the zonotope created by the $\tau$ map in \autoref{lem:bijection_zono_faces_and_arrangement_faces}. On the right we label each face with the root inversion set for that face. See \autoref{ex:Coxeter6} and \autoref{ex:Coxeter7}.}
	\label{fig:A2Roots}
\end{figure}
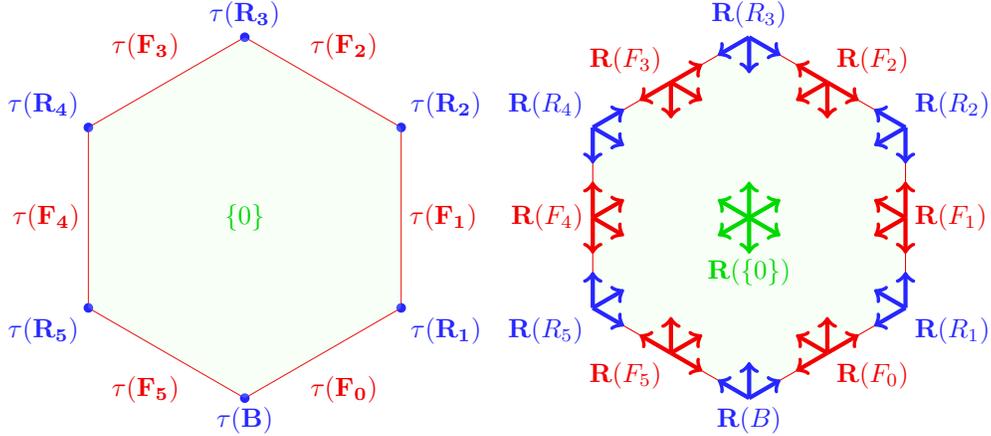

The following lemma shows the relationship between a root being present in a root inversion set and the sign of the covector for the associated hyperplane.
An example of this relationship can be seen in \autoref{fig:A2Covectors} and \autoref{fig:A2Roots}.

\begin{lemma}
    \label{lem:roots_to_covs}
    For any~$F \in \faces$ and $H \in \arrangement$,
    \begin{enumerate}
        \item $F(H) = -$ if and only if $e_H \in \rootInversionSet(F)$ and $-e_H \notin \rootInversionSet(F)$.
        \item $F(H) = 0$ if and only if $e_H \in \rootInversionSet(F)$ and $-e_H \in \rootInversionSet(F)$.
        \item $F(H) = +$ if and only if $e_H \notin \rootInversionSet(F)$ and $-e_H \in \rootInversionSet(F)$.
    \end{enumerate}
    In other words, 
    \[
    	\rootInversionSet(F)^+ = \set{e_H}{H \in \arrangement, \; F(H) \le 0}
		\quad\text{and}\quad
		\rootInversionSet(F)^- = \set{-e_H}{H \in \arrangement, \; F(H) \ge 0}.
    \]
\end{lemma}

\begin{proof}
    We show the first case, the other cases being similar.
    Recall that $e \in \rootInversionSet(F)$ if and only if $\gen{x,e} \leq 0$ for $x \in \int(F)$.
    Furthermore, since $-e \notin \rootInversionSet(F)$ we have~${\gen{x,e} < 0}$.
    By definition of the sign map, since $\gen{x,e} < 0$ we have $\sigma_{H}(x) = -$, \ie $F(H) = -$ as desired.

    Conversely if $F(H) = -$ then $\sigma_{H}(x) = -$ for $x \in \int(F) \subseteq F$.
    Then $\gen{x,e} < 0$ implying that $e \in \rootInversionSet(F)$.
    Furthermore, $\gen{x, -e} > 0$ gives $-e \notin \rootInversionSet(F)$ as desired.
\end{proof}

\begin{corollary}
    \label{cor:faceRoots_contain_pm_root}
    For any $F \in \faces$ and $e \in \Phi_\arrangement$, we have $\rootInversionSet(F) \cap \{e, -e\} \ne \varnothing$.
\end{corollary}

Following up \autoref{thm:equivalence1}, we are now ready to show our second main result, providing two more equivalent descriptions of the facial weak order. 
Recall that for~${X \subseteq \Phi_\arrangement}$, we set~$X^+ \coloneqq X \cap \Phi_\arrangement^+$ and~$X^- \coloneqq X \cap \Phi_\arrangement^-$.

\begin{thm}
	\label{thm:equivalence2}
	The following assertions are equivalent for two faces $F, G \in \faces$:
	\begin{enumerate}
        \item[\autoref{item:fwo_covects}] $F \cole G$ in terms of covectors,
        \addtocounter{enumi}{3}
	    \item \label{item:fwo_geo} $\rootInversionSet(F) \setm \rootInversionSet(G) \subseteq \Phi_\arrangement^-$ and $\rootInversionSet(G) \setm \rootInversionSet(F) \subseteq \Phi_\arrangement^+$,
	    \item \label{item:fwo_geo_2} $\rootInversionSet(F)^+ \subseteq \rootInversionSet(G)^+$ and~$\rootInversionSet(F)^- \supseteq \rootInversionSet(G)^-$.
	\end{enumerate}
\end{thm}

\begin{proof}[Proof of \autoref{thm:equivalence2}]
	The points~\autoref{item:fwo_geo} and \autoref{item:fwo_geo_2} are clearly equivalent.
	We thus just need to prove the equivalence between~\autoref{item:fwo_covects} and~\autoref{item:fwo_geo_2}.

	\medskip\noindent
    \boxed{\textbf{\autoref{item:fwo_covects} $\Rightarrow$ \autoref{item:fwo_geo_2}}}
	Assume that~$F \cole G$ so that~$G(H) \le F(H)$ for all~$H \in \arrangement$ by \autoref{thm:equivalence1}.
	Then for any~$H \in \arrangement$, we obtain by \autoref{lem:roots_to_covs} that
	\begin{itemize}
		\item if~$e_H \in \rootInversionSet(F)$, then~$F(H) \le 0$, so that~$G(H) \le 0$, so that~$e_H \in \rootInversionSet(G)$,
		\item if~$-e_H \in \rootInversionSet(G)$, then~$G(H) \ge 0$, so that~$F(H) \ge 0$, so that~$-e_H \in \rootInversionSet(F)$.
	\end{itemize}
	Therefore~$\rootInversionSet(F)^+ \subseteq \rootInversionSet(G)^+$ and~$\rootInversionSet(F)^- \supseteq \rootInversionSet(G)^-$.
	
	\medskip\noindent
    \boxed{\textbf{\autoref{item:fwo_geo_2} $\Rightarrow$ \autoref{item:fwo_covects}}}
	Assume that~$\rootInversionSet(F)^+ \subseteq \rootInversionSet(G)^+$ and~$\rootInversionSet(F)^- \supseteq \rootInversionSet(G)^-$.
	Then for any~$H \in \arrangement$, we obtain by \autoref{lem:roots_to_covs} that
	\begin{itemize}
		\item if~$G(H) = +$, then~$e_H \notin \rootInversionSet(G)$, so that~$e_H \notin \rootInversionSet(F)$, so that $F(H) = +$,
		\item if~$G(H) = 0$, then~$-e_H \notin \rootInversionSet(G)$, so that~$-e_H \notin \rootInversionSet(F)$, so that $F(H) \ge 0$.
	\end{itemize}
	Therefore~$G(H) \le F(H)$ for all~$H \in \arrangement$, so that~$F \cole G$ as desired.
\end{proof}


\subsection{Zonotopes}
\label{subsec:zonotopes}

We conclude this section with an interpretation of the root inversion sets in terms of the geometry of certain polytopes associated to hyperplane arrangements.

Recall that a \definition{polytope} is the convex hull of finitely many points in~$V$, or equivalently a bounded intersection of finitely many half-spaces of~$V$.
The faces of~$P$ are its intersections with its supporting hyperplanes (and the faces~$\varnothing$ and~$P$ itself), and its facets are its codimension~$1$ faces.
For a face~$F$ of a polytope~$P$, the \definition{inner primal cone} of $F$ is the cone $\pricone(F)$ generated by $\set{u - v}{u \in P\text{, } v \in F}$, and the \definition{outer normal cone} of $F$ is the cone $\norcone(F)$ generated by the outer normal vectors of the facets of~$P$ containing~$F$.
Note that these two cones are dual to one another.
The \definition{normal fan} of~$P$ is the complete polyhedral fan formed using the outer normal cones of all faces of~$P$.
See \cite{Ziegler} for more details.

Here, we still consider a normal vector~$e_H$ to each hyperplane~$H \in \arrangement$ such that the base region~$B$ is contained in the positive half-space~$H^+ = \set{v \in V}{\gen{e_H, v} \geq 0}$.
We are interested in the corresponding zonotope defined below.
Details on zonotopes can be found in the book by G.~M.~Ziegler \cite{Ziegler} and in the article by P.~McMullen~\cite{McMullen}.

\begin{defn}
The \definition{zonotope} $\zonotope$ of the arrangement~$\arrangement$ is the convex polytope 
\[
    \zonotope \coloneqq \biggset{\sum_{H \in \arrangement} \lambda_H e_H}{-1 \le \lambda_H \le 1 \text{ for all } H \in \arrangement}.
\]
\end{defn}

\begin{example}
	\label{ex:Coxeter6}
	The zonotope for a Coxeter arrangement is called a \definition{permutahedron}, see~\cite{Hohlweg}.
	We have represented on the left of \autoref{fig:A2Roots} the zonotope of the arrangement of \autoref{ex:Coxeter1} and \autoref{fig:A2}.
	It has $6$ vertices corresponding to the $6$ regions of the arrangement, and $6$ edges corresponding to the $6$ rays of the arrangement.
\end{example}

Note that this zonotope depends upon the choices of the normal vectors~$e_H$ of the hyperplanes~$H \in \arrangement$, but its combinatorics does not.
Namely, P.~H.~Edelman gives in \cite[Lemma 3.1]{Edelman} a bijection between the nonempty faces of the zonotope~$\zonotope$ and the the faces $\faces$ of the arrangement~$\arrangement$ using the $\tau$ map (given in the following lemma) which was first defined by McMullen in \cite[p. 92]{McMullen}.

\begin{lemma}
    \label{lem:bijection_zono_faces_and_arrangement_faces}
    The map~$\tau$ defined by
    \[
        \tau(F) = \biggset{\smashoperator[r]{\sum_{F \not\subseteq H}} F(H) e_H + \smashoperator[r]{\sum_{F \subseteq H}} \lambda_H e_H}{-1 \le \lambda_H \le 1 \text{ for all } F \subseteq H \in \arrangement}
    \]
    is a bijection from the faces~$\faces$ to the nonempty faces of the zonotope~$\zonotope$.
    Moreover,~$F$ is the outer normal cone~$\norcone \big( \tau(F) \big)$ of~$\tau(F)$, so that the fan of the arrangement~$\arrangement$ is the normal fan of~$\zonotope$.
\end{lemma}

We now relate the root inversion sets of \autoref{subsec:root_inversion_sets} to the faces of the zonotope~$\zonotope$.

\begin{proposition}
    \label{prop:cone_pricone}
    The cone of the root inversion set $\rootInversionSet(F)$ is the inner primal cone of the face $\tau(F)$ in the zonotope $\zonotope$, \ie
    \[
        \cone \big( \rootInversionSet(F) \big) = \pricone \big( \tau(F) \big)
        \quad\text{and}\quad
        \rootInversionSet(F) = \pricone \big( \tau(F) \big) \cap \Phi_\arrangement.
    \]
\end{proposition}

\begin{proof}
    Let $F$ be an arbitrary face in $\faces$ and let $u$ be a point in~$\zonotope$.
    By construction we have $u = \sum_{H \in \arrangement} \lambda_H e_H$ where $\abs{\lambda_H} \leq 1$ for all $H \in \arrangement$.
    Let $v$ be a point in~$\tau(F)$.
    The inner primal cone associated to $F$ in the zonotope $\zonotope$, is~${\pricone \big( \tau(F) \big) = \set{u - v}{u \in \zonotope\text{ and } v \in \tau(F)}}$.
    
    More explicitly, if $H \in \arrangement_F$ then the $e_H$ component of $u - v$ is given by~${(\lambda_H - \lambda'_H) e_H}$ where~${\abs{\lambda_H} \leq 1}$ and $\abs{\lambda'_H} \leq 1$.
    In particular $\pm e_H \in \pricone \big( \tau(F) \big)$.
    If $H \notin \arrangement_F$ then the component of $e_H$ for $u - v$ is given by $(\lambda_H - \mu_H) e_H$ where~${\abs{\lambda_H} \leq 1}$ and~${\mu_H = \pm 1}$.
    Recall from \autoref{lem:bijection_zono_faces_and_arrangement_faces} that $\mu_H = -1$ if $F(H) = -$, etc.
    Suppose~${\mu_H = +1}$, then~${-e_H \in \pricone \big( \tau(F) \big)}$, but $e_H \notin \pricone \big( \tau(F) \big)$.
    Similarly, when~${\mu_H = -1}$, ${e_H \in \pricone \big( \tau(F) \big)}$ and $-e_H \notin \pricone \big( \tau(F) \big)$.
\end{proof}

\begin{example}
	\label{ex:Coxeter7}
    An example of the equality between the cone of the root inversion set with the inner primal cone of the face of the associated zonotope can be seen in \autoref{fig:A2Roots} for the type $A_2$ Coxeter arrangement.
    In \autoref{fig:A2Roots} we have the zonotope~$\zonotope$ on the left and the root inversion set for each face on the right.
    For a face~$F$ of~$\arrangement$, the cone of the root inversion set of $F$ is the same as the inner primal cone of~$\tau(F)$ in~$\zonotope$.
\end{example}


\section{Lattice properties of the facial weak order}
\label{sec:lattice}

It was shown in \cite{DermenjianHohlwegPilaud} that the facial weak order on Coxeter arrangements is a lattice.
The aim of this section is to extend this result to any hyperplane arrangement with a lattice of regions.

\begin{theorem}
	\label{thm:simplicial_is_lattice}
	If $\arrangement$ is an arrangement where $\pr$ is a lattice, then $\fw$ is a lattice.
\end{theorem}

In order to prove this result we use the BEZ lemma which provides a local criterion to characterize finite posets which are lattices.

\begin{lemma}[{\cite[Lemma 2.1]{BjornerEdelmanZiegler}}]
    \label{lem:BEZ}
    If $L$ is a finite, bounded poset such that the join~$x \join y$ exists whenever $x$ and $y$ both cover some $z \in L$, then $L$ is a lattice.
\end{lemma}

So the proof of \autoref{thm:simplicial_is_lattice} reduces to proving the following statement. 

\begin{theorem}
    \label{thm:simplicial_has_join}
 	Let $\arrangement$ be a hyperplane arrangement where $\pr$ is a lattice and let $X, Y,Z$ be three faces of $\arrangement$.
 	If $Z \fwcover X$ and $Z \fwcover Y$, then the join $X \fwjoin Y$ exists.
\end{theorem}

The proof of this theorem is the aim of the next two sections.
The idea of the proof is as follows.
We first consider our cover relations $Z \fwcover X$ and $Z \fwcover Y$.
We know from \autoref{cor:fwo_covers} that this is equivalent to $\abs{\dim Z - \dim X} = 1$, $Z \fwle X$, and either $Z \subseteq X$ or $X \subseteq Z$ and similarly for $Y$. 
By symmetry of~$X$ and~$Y$, we thus obtain the following three cases:
\begin{enumerate}[label=(\arabic*),ref=\arabic*]
    \item \label{it:XY_Z} $X \cup Y \subseteq Z$ and $\dim X = \dim Y = \dim Z - 1$,
    \item \label{it:Z_XY} $Z \subseteq X \cap Y$ and $\dim X = \dim Y = \dim Z + 1$, and
    \item \label{it:X_Z_Y} $X \subseteq Z \subseteq Y$ and $\dim X + 1 = \dim Y - 1 = \dim Z$.
\end{enumerate}
In each case we consider the subarrangement associated to the largest face contained in all three faces.
Namely, the subarrangement~${\arrangement_{X \cap Y} = \set{H \in \arrangement}{X \cap Y \subseteq H}}$ for case \autoref{it:XY_Z}, the subarrangement~${\arrangement_Z = \set{H \in \arrangement}{Z \subseteq H}}$  for case \autoref{it:Z_XY} and the subarrangement~${\arrangement_X = \set{H \in \arrangement}{X \subseteq H}}$ for case \autoref{it:X_Z_Y} .

In the next subsection we show that the join in the poset of regions of a subarrangement can be extended to a join in the poset of regions of the arrangement itself.
Finally, for each case we find the join inside the appropriate subarrangement, culminating in the proof of \autoref{thm:simplicial_has_join}.

Before we begin, we give a conjecture stating that the converse of \autoref{thm:simplicial_is_lattice} is true as well.

\begin{conjecture}
    \label{conj:lattice_iff_lattice}
    For any hyperplane arrangement~$\arrangement$ and any base region~$B$ of~$\arrangement$, the poset of regions~$\pr$ is a lattice if and only if the facial weak order~$\fw$ is a lattice.
\end{conjecture}


\subsection{Joins and subarrangements of faces}
\label{subsec:hyperplane_subarrangements}

A \definition{subarrangement} of an arrangement~$\arrangement$ is a subset~$\arrangement'$ of~$\arrangement$.
There is a natural map~$\faces \to \faces[\arrangement']$ that projects each face $G$ in~$\faces$ to the smallest face $G_{\arrangement'}$ in $\faces[\arrangement']$ such that the relative interior of $G$ is contained in the relative interior of $G_{\arrangement'}$, \ie for $H \in \arrangement'$ then $G_{\arrangement'}(H) = G(H)$.
Note that this map is surjective and preserves the facial weak order: if $F \fwle G$ in~$\arrangement$, then~$F_{\arrangement'} \fwle G_{\arrangement'}$ in~$\arrangement'$.

We particularly focus on the following special subarrangements.
For a face~${F \in \faces}$, let $\arrangement_F \coloneqq \set{H \in \arrangement}{F \subseteq H}$ be the subarrangement of $\arrangement$ with all hyperplanes which contain $F$.
This subarrangement~$\arrangement_F$ is known as the \definition{support} of $F$ or the \definition{localization} of $\arrangement$ to $F$.
We denote by~$\pi_F$ the projection map ${\faces \to \faces[\arrangement_F]}$ described above in this specific case, and we often use the shorthand~$G_F$ for~${\pi_F(G) = G_{\arrangement_F}}$.
Note that the surjection~$\pi_F$ restricts to a bijection between $\set{G \in \faces}{F \subseteq G}$ and~$\faces[\arrangement_F]$.

\begin{example}
	\label{ex:Coxeter8}
    \autoref{fig:A2Map} gives an example of these maps for the subarrangement~$\arrangement_{F_1}$ of the type~$A_2$ arrangement discussed in \autoref{ex:Coxeter1}.
    Since $H_2$ is the only hyperplane containing~$F_1$, our subarrangement contains one hyperplane ${\arrangement_{F_1} = \{H_2\}}$.
    Then $\pi_{F_1}: \faces \to \faces[\arrangement_{F_1}]$ is the map with the following equalities (some of which are shown in the figure, but not all).
    \begin{align*}
        \pi_{F_1}(R_2) &= \pi_{F_1}(F_2) = \pi_{F_1}(R_3) = \pi_{F_1}(F_3) = \pi_{F_1}(R_4)\\
        \pi_{F_1}(F_1) &= \pi_{F_1}(0) = \pi_{F_1}(F_4)\\
        \pi_{F_1}(R_1) &= \pi_{F_1}(F_0) = \pi_{F_1}(B) = \pi_{F_1}(F_5) = \pi_{F_1}(R_5)
    \end{align*}
    It can be seen in the figure that $\pi_{F_1}$ is a bijection from $\{R_1,R_2,F_1\}$ to~$\arrangement_{F_1}$.
	\begin{figure}[b]
\DeclareDocumentCommand{\rs}{ O{1.1cm} O{->} m m O{0} O{0}} {
	\def \radius {#1}
	\def \inputPoints{#3}
	\def \excludeRoots{#4}
	\def \style {#2}
	\def \initialRotation {#5}
	\def \extraStyle {#6}

	\pgfmathtruncatemacro{\points}{\inputPoints * 2}
	\pgfmathsetmacro{\degrees}{360 / \points}
	
	\coordinate (0) at (0,0);
	
	\foreach \x in {1,...,\points}{%
		\pgfmathsetmacro{\location}{(\points+(\x-1))*\degrees + \initialRotation}
		
		\coordinate (\x) at (\location:\radius);
	}

	\ifthenelse{\equal{\excludeRoots}{}}{
		\foreach \x in {1,...,\points}{%
			\ifthenelse{\equal{\extraStyle}{0}}{
				\draw[\style] (0) -- (\x);
			}{
			\draw[\style, \extraStyle] (0) -- (\x);
		}
	}
}{
\foreach \x in {1,...,\points}{%
	\edef \showPoint {1};
	
	\foreach \y in \excludeRoots {
		\ifthenelse{\equal{\x}{\y}}{
			\xdef \showPoint {0};
		}{}
	}
	
	\ifthenelse{\equal{\showPoint}{1}}{
		\ifthenelse{\equal{\extraStyle}{0}}{
			\draw[\style] (0) -- (\x);
		}{
		\draw[\style, \extraStyle] (0) -- (\x);
	}
}{}
}
}  
}

\centerline{
    \footnotesize
    \begin{tikzpicture}
	    [
            scale=0.6,
	    fdomain/.style={fill=blue!15!white,color=blue!15!white,opacity=0.3},
	    vdomain/.style={color=blue!85!white},
	    face/.style={draw=red!95!black,fill=red!95!black,color=red!95!black,ultra thick, -},
        vface/.style={color=red!95!black},
	    norm/.style={black,->},
	    vertex/.style={inner sep=1pt,circle,draw=green!85!white,fill=green!85!white,thick},
	    ]
        \begin{scope}[shift={(0,0)}]
    	    %
    	    \rs[3.5][ultra thick]{3}{}[60]
    	    \node[above right] at (1) {$H_{3}$};
    	    \node[above left] at (2) {$H_{1}$};
    	    \node[left] at (3) {$H_{2}$};
    	    \rs[3][dashed]{3}{}[240]
    	    \draw[face] (0) -- (2) node[right] {$\mathbf{F_0}$};
    	    \draw[face] (0) -- (3) node[above] {$\mathbf{F_1}$};
    	    \draw[face] (0) -- (4) node[right] {$\mathbf{F_2}$};
    	    \draw[face] (0) -- (5) node[left] {$\mathbf{F_3}$};
    	    \draw[face] (0) -- (6) node[above] {$\mathbf{F_4}$};
    	    \draw[face] (0) -- (1) node[left] {$\mathbf{F_5}$};
    	    %
    	    %
    	    \fill[fdomain] (0) -- (1) -- (2) -- cycle {};
    	    \fill[fdomain] (0) -- (3) -- (2) -- cycle {};
    	    \fill[fdomain] (0) -- (3) -- (4) -- cycle {};
    	    \fill[fdomain] (0) -- (5) -- (4) -- cycle {};
    	    \fill[fdomain] (0) -- (5) -- (6) -- cycle {};
    	    \fill[fdomain] (0) -- (1) -- (6) -- cycle {};
    	    \node[vdomain] at (270:2.2) {$B$};
    	    \node[vdomain] at (330:2.2) {$R_1$};
    	    \node[vdomain] at (30:2.2) {$R_2$};
    	    \node[vdomain] at (90:2.2) {$R_3$};
    	    \node[vdomain] at (150:2.2) {$R_4$};
    	    \node[vdomain] at (210:2.2) {$R_5$};
    	    \node[vertex] at (0) {};
        \end{scope}
        \begin{scope}[shift={(8,0)}]
    	    %
            \draw[ultra thick] (0:3.5) -- (180:3.5);
            \node[right] at (0:3.5) {$H_{2}$};
            \draw[face] (180:3) -- (0:3);
            \node[vface,above] at (0,0) {$\pi_{F_1}(0)$};
    	    %
            \draw[fdomain] (0,0) circle (3); 
            \node[vdomain] at (270:2.2) {$\pi_{F_1}(B) = \pi_{F_1}(F_0)$};
            \node[vdomain] at (90:2.2) {$\pi_{F_1}(R_3) = \pi_{F_1}(F_3)$};
        \end{scope}
    \end{tikzpicture}
}
	    \caption{The map $\pi_{F_1}$ from an arrangement $\arrangement$ to a subarrangement $\arrangement_{F_1}$. See \autoref{ex:Coxeter8}.}
		\label{fig:A2Map}
	\end{figure}
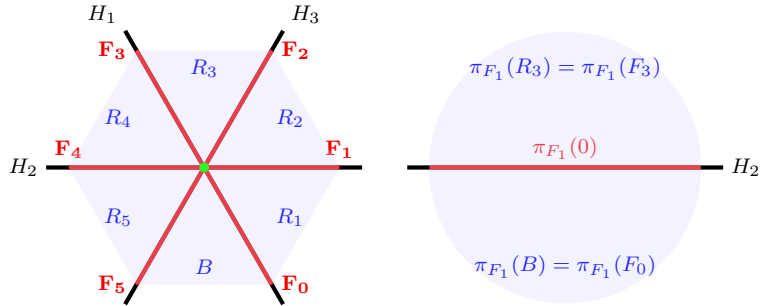
\end{example}

Given an arrangement $\arrangement$ whose poset of regions $\pr$ is a lattice, it is not necessary that any arbitrary subarrangement will also have a lattice of regions.
However, when the subarrangement is associated to a face, then the lattice property of the poset of regions is preserved through facial intervals.
This follows from the well-known fact that an interval of a lattice is a lattice.
This lattice property, combined with the fact that the base region of a lattice of regions is always simplicial (see \cite[Theorem 3.1 and 3.4]{Edelman}) gives the following proposition.

\begin{proposition}
    \label{prop:facial_int_are_simplicial}
    Let $\arrangement$ be an arrangement whose poset of regions $\pr$ is a lattice.
    For a face $F \in \faces$ the subarrangement $\arrangement_F$ is a central subarrangement and~$\pr[\arrangement_F][B_F]$ is a lattice of regions with simplicial base region~$B_F$.
\end{proposition}

\begin{lemma}
	\label{lem:localization}
    For any three faces $X, Y, Z \in \faces$ such that $[X,Y]$ is an interval in~$\fw$ and $Z \subseteq X \cap Y$, then the interval $[X,Y]$ in $\fw$ is isomorphic to~$[X_{Z},Y_{Z}]$ in $\fw[\arrangement_{Z}][B_{Z}]$.
\end{lemma}

\begin{proof}
    We first prove that the map $W \mapsto W_{Z}$ defines an injective order preserving map from $[X,Y]$ to $[X_{Z},Y_{Z}]$.
    Let $W \in [X,Y]$.
    We aim to show $Z$ is a face of $W$.
    By \autoref{prop:face_inclusion_on_covectors} it suffices to show $Z(H) = W(H)$ when $Z(H) \neq 0$.
    Suppose that there is $H \in \arrangement$ such that $Z(H) \ne W(H)$ and $Z(H) \ne 0$.
    This implies ${Z(H) = X(H) = Y(H)}$.
    As $W \in [X,Y]$, then ${Y(H) \leq W(H) \leq X(H)}$.
    Hence $W(H) = X(H) = Z(H)$, a contradiction.
    Therefore, $Z$ is a face of $W$.
    Thus, the localization map ${[X,Y]\rightarrow[X_{Z},Y_{Z}]}$ is injective.

    The inverse map $[X_{Z},Y_{Z}]\rightarrow[X,Y]$ is defined by extending $W_{Z}$ to a covector $W$ with $W(H)=Z(H)$ for $H\in\arrangement\setm\arrangement_{Z}$.
    As this map is also order preserving, the proof is complete.
\end{proof}

This lemma gives us another way to view facial intervals as the faces of a subarrangement.
With the above lemma, given a facial interval $[m_F, M_F]$ for a face $F$, then $m_F$ (resp.~$M_F$) is the region in $\arrangement$ associated to the base region ${B}_F$   (resp.~to its opposite region $-{B}_F$) in $\faces[\arrangement_F]$.
We now show that a join in the poset of regions of a subarrangement extends to a join in the poset of regions of the arrangement itself.
The following is possible by \autoref{prop:facial_int_are_simplicial}.

\begin{proposition}
	\label{prop:look_inside_subarr}
    For any three faces~$X, Y, Z \in \faces$ such that $Z \subseteq X \cap Y$, if there exists a face~$W$ containing~$Z$ such that $W_Z = X_Z \fwjoin Y_Z$ in $\fw[\arrangement_Z][B_Z]$ then~${W = X \fwjoin Y}$ in $\fw$.
\end{proposition}

\begin{proof}
    Suppose $U \in \faces$ is a face such that $X \fwle U$ and $Y \fwle U$.
    Since the projection map $\pi_F: \faces \to \faces[\arrangement_F]$ preserves the facial weak order, we have that~${W_Z = X_Z \fwjoin Y_Z \fwle U_Z}$ in the facial weak order of the subarrangement.
	In other words, for all $H \in \arrangement_Z$, we have $U_Z(H) \leq W_Z(H)$, and thus $U(H) \leq W(H)$.
	
	Next let $H'$ be a hyperplane in $\arrangement \setm \arrangement_Z$.
    Since $H' \notin \arrangement_Z$, we have~${Z(H') \ne 0}$.
    Furthermore, by \autoref{prop:face_inclusion_on_covectors}, $0 \ne Z(H') = X(H') = Y(H') = W(H')$ since~${Z \subseteq X \cap Y}$ and $Z \subseteq W$.
    Then, since $X \fwle U$, we have~${U(H') \leq X(H') = W(H')}$.

	In other words, $U(H) \leq W(H)$ for all $H \in \arrangement$.
    Therefore $W \fwle U$ implying~${W =  X \fwjoin Y}$.
\end{proof}


\subsection{Joins in subarrangements}

As discussed, we now describe the three distinct cases that arise using the cover relations of the facial weak order.
Then, for each case, we restrict ourselves to the subarrangement associated to the largest face contained in all three faces and find the join in the subarrangement.
Combining these results with \autoref{prop:look_inside_subarr} proves \autoref{thm:simplicial_has_join}.

Consider three faces~$X, Y, Z \in \faces$ such that~$Z \fwcover X$ and $Z \fwcover Y$.
Recall that by \autoref{cor:fwo_covers}, we have $Z \fwcover X$ if and only if $\abs{\dim Z - \dim X} = 1$, $Z \fwle X$, and either $Z \subseteq X$ or $X \subseteq Z$, and similarly for $Y$.
By symmetry on~$X$ and~$Y$, this gives us three different cases:
\begin{enumerate}[label=(\arabic*)]
	\item $X \cup Y \subseteq Z$ and $\dim X = \dim Y = \dim Z - 1$,
	\item $Z \subseteq X \cap Y$ and $\dim X = \dim Y = \dim Z + 1$, and
	\item $X \subseteq Z \subseteq Y$ and $\dim X +1 = \dim Y - 1 = \dim Z$.
\end{enumerate}
We now look at each case individually.
We have broken down their proofs into three subsections to better facilitate their reading.
We let $\boundary(R)$ denote the set of boundary hyperplanes of a region $R$.


\subsubsection{First case: $X \cup Y \subseteq Z$ and $\dim X = \dim Y = \dim Z - 1$}
\label{subsubsec:first_case}

\begin{figure}[b]
\DeclareDocumentCommand{\rs}{ O{1.1cm} O{->} m m O{0} O{0}} {
	\def \radius {#1}
	\def \inputPoints{#3}
	\def \excludeRoots{#4}
	\def \style {#2}
	\def \initialRotation {#5}
	\def \extraStyle {#6}

	\pgfmathtruncatemacro{\points}{\inputPoints * 2}
	\pgfmathsetmacro{\degrees}{360 / \points}
	
	\coordinate (0) at (0,0);
	
	\foreach \x in {1,...,\points}{%
		\pgfmathsetmacro{\location}{(\points+(\x-1))*\degrees + \initialRotation}
		
		\coordinate (\x) at (\location:\radius);
	}

	\ifthenelse{\equal{\excludeRoots}{}}{
		\foreach \x in {1,...,\points}{%
			\ifthenelse{\equal{\extraStyle}{0}}{
				\draw[\style] (0) -- (\x);
			}{
			\draw[\style, \extraStyle] (0) -- (\x);
		}
	}
}{
\foreach \x in {1,...,\points}{%
	\edef \showPoint {1};
	
	\foreach \y in \excludeRoots {
		\ifthenelse{\equal{\x}{\y}}{
			\xdef \showPoint {0};
		}{}
	}
	
	\ifthenelse{\equal{\showPoint}{1}}{
		\ifthenelse{\equal{\extraStyle}{0}}{
			\draw[\style] (0) -- (\x);
		}{
		\draw[\style, \extraStyle] (0) -- (\x);
	}
}{}
}
}  
}

\centerline{
    \begin{tikzpicture}
	    [
	    fdomain/.style={fill=blue!15!white,color=blue!15!white,opacity=0.3},
	    vdomain/.style={color=blue!85!white},
	    face/.style={draw=red!95!black,fill=red!95!black,color=red!95!black,ultra thick, -},
	    vertex/.style={inner sep=1pt,circle,draw=green!85!white,fill=green!85!white,thick},
	    vtext/.style={color=green!55!black},
	    ]
	    %
	    %
	    %
        \rs[3.5][ultra thick]{4}{}[22.5]
	    \node[above left] at (2) {$H_{2}$};
	    \node[above right] at (3) {$H_{1}$};
        \rs[3][dashed]{4}{}[22.5]
        \draw[face] (0) -- (7) node[right] {$Y$};
        \draw[face] (0) -- (6) node[left] {$X$};
	    %
	    %
	    \fill[fdomain] (0) -- (6) -- (7) -- cycle {};
	    \node[] at (180:2.2) {$\vdots$};
	    \node[] at (0:2.2) {$\vdots$};
	     \node[vdomain] at (270:2.2) {$B=Z$};
	    \node[vertex] at (0) {};
        \node[vtext] at (0:1.2) {$0 = X \cap Y$};
    \end{tikzpicture}
}
    \caption{The construction of the join for the first case when $X \cup Y \subseteq Z$.}
	\label{fig:Join_XY_in_Z}
\end{figure}
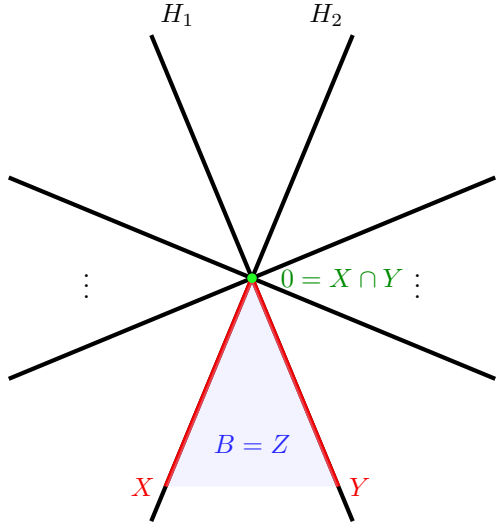

Since $X \cap Y$ is the largest face contained in $X$, $Y$ and~$Z$, we restrict to the subarrangement~${\arrangement_{X \cap Y}}$ and find the join there.
An example (in rank $2$) is given in \autoref{fig:Join_XY_in_Z}.
By \autoref{prop:facial_int_are_simplicial}, the poset of regions $\pr[\arrangement_{X \cap Y}][B_{X \cap Y}]$ is a lattice.
Thus, without loss of generality, it suffices to prove the following proposition.

\begin{proposition}
	\label{prop:XY_Z_join}
    Consider an arrangement $\arrangement$ whose poset of regions is a lattice with three faces $X$, $Y$ and $Z$ such that $Z \fwcover X$, $Z \fwcover Y$, $\{0\} = X \cap Y$ and~${X \cup Y \subseteq Z}$.
	Then $\{0\} = X \cap Y = X \fwjoin Y$.
\end{proposition}

\begin{proof}
	We first prove that~$X \fwle X \cap Y = \{0\}$.
	Assume by contradiction that there is~$H \in \arrangement$ such that~$X(H) = -$.
	Since $X \subseteq Z$, we obtain that $Z(H) = -$ by \autoref{prop:face_inclusion_on_covectors}.
    Moreover, since $Z \fwcover Y$, we have $Y(H) \leq Z(H) = -$.

    Let $[m_Z, M_Z]$ be the facial interval in the poset of regions associated to $Z$.
    As $\dim X = \dim Y = \dim Z - 1$, there exist boundary hyperplanes $H_1$ and $H_2$ of~$m_Z$ such that $X = Z \cap H_1$ and $Y = Z \cap H_2$.
    Since~$X \ne Z \ne Y$, we obtain by \autoref{lem:Y_(_X_gives_nonzero_X*(H)} that~$Z(H_1) \ne 0 \ne Z(H_2)$.
    Since~$0 = X(H_1) \le Z(H_1)$ and $0 = Y(H_2) \le Z(H_2)$, we conclude that~$Z(H_1) = Y(H_1) = +$ and $Z(H_2) = X(H_2) = +$.
    
    Let $\arrangement' \coloneqq \{H_1, H_2, H\}$ be the subarrangement of $\arrangement$ with these three hyperplanes.
    Since $Z(H_1) = Z(H_2) = +$ and $Z(H) = -$, the face $Z_{\arrangement'}$ is a region in $\arrangement'$.
    Moreover, we have~$(X \cap Y)_{\arrangement'} (H) = 0$ because~$(X \cap Y)(H) = 0$.
    We thus obtain that~$H_1$ and~$H_2$ are the only boundary hyperplanes of~$Z_{\arrangement'}$ in~$\arrangement'$.
    Therefore, for any region~${R \in \regions[\arrangement'] \setm \left\{ Z_{\arrangement'} \right\}}$, then either $H_1$ or $H_2$ is in the separation set $S(R, Z_{\arrangement'})$.
    Since~${Z_{\arrangement'}(H_1) = Z_{\arrangement'}(H_2) = +}$, then either $R(H_1) = -$ or $R(H_2) = -$.
    It follows that no region of~$\arrangement'$ is all positive, a contradiction since the base region~$B_{\arrangement'}$ is all positive.

    We conclude that~$X(H) \ge 0$ for all~${H \in \arrangement}$, so that~${X \fwle X \cap Y = \{0\}}$.
	By symmetry, we also obtain that~$Y \fwle X \cap Y = \{0\}$.

    Finally, to prove that $\{0\} = X \cap Y = X \fwjoin Y$, we consider an arbitrary face~$U$ in $\arrangement$ such that $X \fwle U$ and $Y \fwle U$.
    Then, $U(H) \le \min \big( X(H), Y(H) \big)$ for all~${H \in \arrangement}$.
    Since $(X \cap Y)(H) =  0 \le \min \big( X(H), Y(H) \big)$ for all $H \in \arrangement$, the faces $X$ and $Y$ of~$\faces$ are contained in $B$ by \autoref{cor:containment_B}.
    By \autoref{prop:facial_int_are_simplicial},~$B$ is simplicial and therefore, there exists $H_3$, $H_4$ in $\boundary(B)$ such that $X \cap Y = X \cap H_3 = Y \cap H_4$ with $X(H) = Y(H) = 0$ for all~${H \in \boundary(B)\setm \{H_3, H_4\}}$.
    Note that $H_3$ and $H_4$ could be the same $H_1$ and $H_2$ as before.
    Therefore, $0 = \min \big( X(H), Y(H) \big)$ for all~${H \in \boundary(B)}$.
    We conclude that~$U(H) \le 0$ for all $H \in \boundary(B)$.
    By \autoref{cor:containment_B},~${U(H) \leq 0}$ for all $H \in \arrangement$.
    Therefore, $X \cap Y \fwle U$.
\end{proof}

%
%


\subsubsection{Second case: $Z \subseteq X \cap Y$ and $\dim X = \dim Y = \dim Z + 1$}
\label{subsubsec:second_case}

\begin{figure}[b]
\DeclareDocumentCommand{\rs}{ O{1.1cm} O{->} m m O{0} O{0}} {
	\def \radius {#1}
	\def \inputPoints{#3}
	\def \excludeRoots{#4}
	\def \style {#2}
	\def \initialRotation {#5}
	\def \extraStyle {#6}

	\pgfmathtruncatemacro{\points}{\inputPoints * 2}
	\pgfmathsetmacro{\degrees}{360 / \points}
	
	\coordinate (0) at (0,0);
	
	\foreach \x in {1,...,\points}{%
		\pgfmathsetmacro{\location}{(\points+(\x-1))*\degrees + \initialRotation}
		
		\coordinate (\x) at (\location:\radius);
	}

	\ifthenelse{\equal{\excludeRoots}{}}{
		\foreach \x in {1,...,\points}{%
			\ifthenelse{\equal{\extraStyle}{0}}{
				\draw[\style] (0) -- (\x);
			}{
			\draw[\style, \extraStyle] (0) -- (\x);
		}
	}
}{
\foreach \x in {1,...,\points}{%
	\edef \showPoint {1};
	
	\foreach \y in \excludeRoots {
		\ifthenelse{\equal{\x}{\y}}{
			\xdef \showPoint {0};
		}{}
	}
	
	\ifthenelse{\equal{\showPoint}{1}}{
		\ifthenelse{\equal{\extraStyle}{0}}{
			\draw[\style] (0) -- (\x);
		}{
		\draw[\style, \extraStyle] (0) -- (\x);
	}
}{}
}
}  
}

\centerline{
    \begin{tikzpicture}
	    [
	    fdomain/.style={fill=blue!15!white,color=blue!15!white,opacity=0.3},
	    vdomain/.style={color=blue!85!white},
	    face/.style={draw=red!95!black,fill=red!95!black,color=red!95!black,ultra thick, -},
	    vertex/.style={inner sep=1pt,circle,draw=green!85!white,fill=green!85!white,thick},
	    vtext/.style={color=green!55!black},
	    ]
	    %
	    %
	    %
        \rs[3.5][ultra thick]{4}{}[22.5]
	    \node[above left] at (2) {$H_{1}$};
	    \node[above right] at (3) {$H_{2}$};
        \rs[3][dashed]{4}{}[22.5]
        \draw[face] (0) -- (3) node[left] {$Y$};
        \draw[face] (0) -- (2) node[right] {$X$};
	    %
	    %
	    \fill[fdomain] (0) -- (2) -- (3) -- cycle {};
        \node[vdomain] at (90:2.2) {$W$};
	    \node[] at (180:2.2) {$\vdots$};
	    \node[] at (0:2.2) {$\vdots$};
	     \node[vdomain] at (270:2.2) {$B$};
	    \node[vertex] at (0) {};
        \node[vtext] at (0:.9) {$0 = Z$};
    \end{tikzpicture}
}
    \caption{The construction of the join for the second case when $Z \subseteq X \cap Y$.}
	\label{fig:Join_Z_in_XY}
\end{figure}
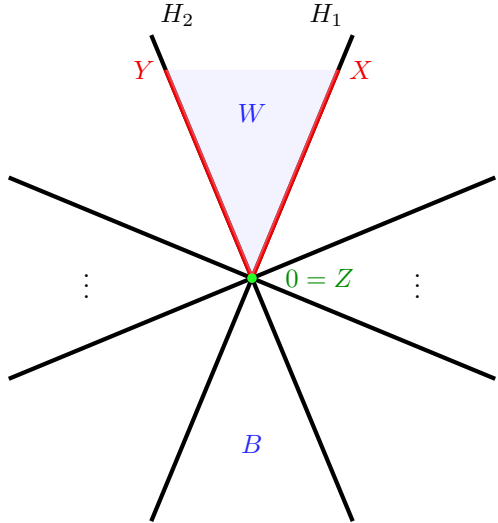

Since $Z$ is the largest face contained in~$X$, $Y$ and~$Z$, we restrict to the subarrangement $\arrangement_Z$ and find the join there.
An example (in rank $2$) is given in \autoref{fig:Join_Z_in_XY}.
By \autoref{prop:facial_int_are_simplicial}, the poset of regions $\pr[\arrangement_Z][B_Z]$ is a lattice.
Therefore, without loss of generality, we consider an arrangement $\arrangement$ whose poset of regions is a lattice with distinct faces~$X$, $Y$, and $Z = \{0\}$ such that $\{0\} = Z \fwcover X$ and $\{0\} = Z \fwcover Y$.
Observe that this implies by \autoref{cor:containment_B} that~$X$ and~$Y$ are rays of the region~$-B$ opposite to the base region~$B$.
Since~$\pr$ is a lattice, $-B$ is simplicial, and therefore there is a $2$-dimensional face~$W$ of~$-B$ containing both $X$ and~$Y$.
This gives us the join of~$X$ and~$Y$.

\begin{proposition}
	\label{prop:Z_XY_join}
    Consider an arrangement $\arrangement$ whose poset of regions is a lattice with distinct faces $X$, $Y$, and~${Z = \{0\}}$ such that $\{0\} = Z \fwcover X$ and~${\{0\} = Z \fwcover Y}$.
	Then~$X \fwjoin Y = W$ where~$W$ is the $2$-dimensional face of~$-B$ containing both $X$ and~$Y$.
\end{proposition}

\begin{proof}
	Since~$X \subseteq W$ are all faces of~$-B$, we have~$X \fwle W$ by \autoref{cor:fwle_B}.
	Similarly, we have~$Y \fwle W$.
	
	Conversely, consider a face~$U \in \faces$ such that $X \fwle U$ and $Y \fwle U$.
    For any~$H \in \arrangement$, we have~$X(H) \le 0$ by \autoref{cor:containment_B} since~$X$ is a face of~$-B$.
    Since~$X \fwle U$, we have $U(H) \le X(H) \le 0$ for all $H \in \arrangement$, which implies that~$U$ is a face of $-B$ by \autoref{cor:containment_B}.
    Since~$X \fwle U$, \autoref{cor:fwle_B} implies that~$X \subseteq U$.
    Similarly, $Y \subseteq U$ and thus~$W \subseteq U$.
    We conclude that~$W \fwle U$ by \autoref{cor:fwle_B}.
\end{proof}


\subsubsection{Third case: $X \subseteq Z \subseteq Y$  and $\dim X + 1 = \dim Y - 1 = \dim Z$}
\label{subsubsec:third_case}

\begin{figure}[b]
\DeclareDocumentCommand{\rs}{ O{1.1cm} O{->} m m O{0} O{0}} {
	\def \radius {#1}
	\def \inputPoints{#3}
	\def \excludeRoots{#4}
	\def \style {#2}
	\def \initialRotation {#5}
	\def \extraStyle {#6}

	\pgfmathtruncatemacro{\points}{\inputPoints * 2}
	\pgfmathsetmacro{\degrees}{360 / \points}
	
	\coordinate (0) at (0,0);
	
	\foreach \x in {1,...,\points}{%
		\pgfmathsetmacro{\location}{(\points+(\x-1))*\degrees + \initialRotation}
		
		\coordinate (\x) at (\location:\radius);
	}

	\ifthenelse{\equal{\excludeRoots}{}}{
		\foreach \x in {1,...,\points}{%
			\ifthenelse{\equal{\extraStyle}{0}}{
				\draw[\style] (0) -- (\x);
			}{
			\draw[\style, \extraStyle] (0) -- (\x);
		}
	}
}{
\foreach \x in {1,...,\points}{%
	\edef \showPoint {1};
	
	\foreach \y in \excludeRoots {
		\ifthenelse{\equal{\x}{\y}}{
			\xdef \showPoint {0};
		}{}
	}
	
	\ifthenelse{\equal{\showPoint}{1}}{
		\ifthenelse{\equal{\extraStyle}{0}}{
			\draw[\style] (0) -- (\x);
		}{
		\draw[\style, \extraStyle] (0) -- (\x);
	}
}{}
}
}  
}

\centerline{
    \begin{tikzpicture}
	    [
	    fdomain/.style={fill=blue!15!white,color=blue!15!white,opacity=0.3},
	    vdomain/.style={color=blue!85!white},
	    face/.style={draw=red!95!black,fill=red!95!black,color=red!95!black,ultra thick, -},
	    vertex/.style={inner sep=1pt,circle,draw=green!85!white,fill=green!85!white,thick},
	    vtext/.style={color=green!55!black},
	    ]
	    %
	    %
	    %
        \rs[3.5][ultra thick]{4}{}[22.5]
	    \node[above left] at (2) {$H_{2}$};
	    \node[above right] at (3) {$H_{1}$};
        \rs[3][dashed]{4}{}[22.5]
	    \draw[face] (0) -- (6) node[above left] {$Z$};
        \draw[face] (0) -- (7) node[above right] {$W$};
        \draw[face] (0) -- (3) node[left] {$-W$};
	    %
	    %
	    \fill[fdomain] (0) -- (6) -- (7) -- cycle {};
        \node[vdomain] at (270:2.2) {$Y_{-Z}$};
	    \fill[fdomain] (0) -- (5) -- (6) -- cycle {};
	    \node[vdomain] at (225:2.2) {$Y$};
	    \node[] at (180:2.2) {$\vdots$};
	    \node[] at (0:2.2) {$\vdots$};
	    \node[vertex] at (0) {};
        \node[vtext] at (0:.9) {$0 = X$};
    \end{tikzpicture}
}
    \caption{The construction of the join in the third case when $X \subseteq Z \subseteq Y$.}
	\label{fig:Join_X_in_Z_in_Y}
\end{figure}
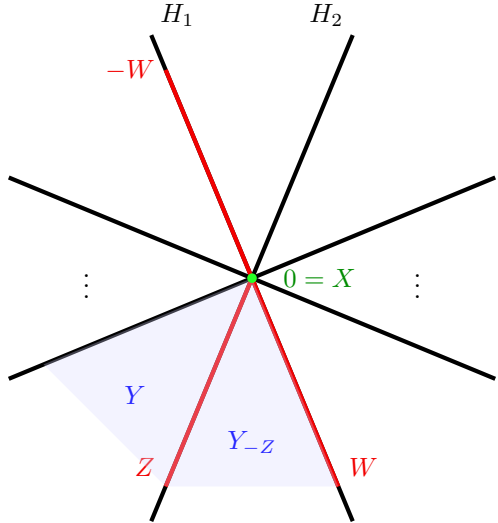

Since $X$ is the largest face contained in~$X$, $Y$ and $Z$, we restrict to the subarrangement~$\arrangement_X$ and find the join there.
An example (in rank $2$) is given in \autoref{fig:Join_X_in_Z_in_Y}.
By \autoref{prop:facial_int_are_simplicial}, the poset of regions $\pr[\arrangement_X][B_X]$ is a lattice.
Therefore, without loss of generality, we consider an arrangement $\arrangement$ whose poset of regions is a lattice with three faces~$X = \{0\}$, $Y$ and~$Z$ such that $Z \fwcover X = \{0\}$, $Z \fwcover Y$ and~${\{0\} = X \subseteq Z \subseteq Y}$.
Observe that this implies that~$Z$ is a ray of the base region~$B$ by \autoref{cor:containment_B}.
Remember that for two faces~$F$ and~$G$, we denote by~$F_{-G}$ the reorientation of~$F$ by~$G$ (see \autoref{subsec:covectors}).
Observe that~$Z$ is a ray of the $2$-dimensional cone~$Y_{-Z}$, and let~$W$ denote its other ray.
We aim to prove the following proposition.

\begin{proposition}
    \label{prop:X_Z_Y_join}
    Consider an arrangement $\arrangement$ whose poset of regions is a lattice with three faces~$X = \{0\}$, $Y$ and~$Z$ such that $Z \fwcover X = \{0\}$, $Z \fwcover Y$ and~${\{0\} = X \subseteq Z \subseteq Y}$.
    Then $X \fwjoin Y = -W$ where~$W$ is the ray of~$Y_{-Z}$ distinct from~$Z$.
\end{proposition}

We will prove that $X \fwjoin Y = -W$ in \autoref{lemma:W-X_is_join_cand} and \autoref{lemma:W-X_is_join}.
We first identify two crucial boundary hyperplanes of the base region~$B$.

\begin{lemma}
    \label{lem:def_H1_H2}
    There exists two unique boundary hyperplanes $H_1$ and $H_2$ of the base region~$B$ such that $\{0\} = X = Z \cap H_1$ and~$Z = Y \cap H_2$.
\end{lemma}

\begin{proof}
    As $Z$ is a ray of the (simplicial) base region~$B$, there is a unique~${H_1 \in \boundary(B)}$ such that~$\{0\} = X = Z \cap H_1$.
	For the second hyperplane, we first claim that there is a unique boundary hyperplane~$H_2$ of the base region such that~$Y(H_2) = -$ while $Z(H_2) = 0$.
	Indeed, if there were two such hyperplanes~$H_2$ and~$H_2'$, we would have~$Z \subsetneq Y \cap H_2 \subsetneq Y$ contradicting that~$\dim Y = \dim Z + 1$.
	Moreover, since~$Z \fwle Y$, there is no hyperplane~$H$ such that~$Y(H) = +$ and~$Z(H) = 0$.
	We conclude that~$H_2$ is the unique hyperplane of~$\boundary(B)$ such that~$Z = Y \cap H_2$.
\end{proof}

\begin{lemma}
    \label{lem:covec_Y_Z_W_W-X_comparisons}
    Consider the two boundary hyperplanes~$H_1$ and~$H_2$ of the base region given in \autoref{lem:def_H1_H2}. Then
    \begin{align*}
        & 0 = W(H_1) = X(H_1) < Y(H_1) = Z(H_1) = +\text{, }\\
        & - = Y(H_2) < X(H_2) = 0 = Z(H_2) < W(H_2) = + \text{, and} \\
        & 0 = W(H) = X(H) = Y(H) = Z(H) \text{ for all } H \in \boundary(B) \setm \{H_1,H_2\}.
    \end{align*}
\end{lemma}

\begin{proof}
	Since~$X = \{0\}$, we have~$X(H) = 0$ for all~$H \in \boundary(B)$.
    By definition of~$H_1$ and~$H_2$ and \autoref{lem:Y_(_X_gives_nonzero_X*(H)}, we have $X(H_1) = 0 \ne Z(H_1)$ and~${Z(H_2) = 0 \ne Y(H_2)}$.
    Since~$Z \fwle X$ and $Z \fwle Y$, this implies that~$Z(H_1) = +$ and~${Y(H_2) = -}$.
	Moreover, as~$Z$ is a face of~$Y$, we obtain that~$Y(H_1) = +$.
    Finally, for any hyperplane~${H \in \boundary(B) \setm \{H_1,H_2\}}$, we have~$Z(H) = 0$ by uniqueness of~$H_2$, and therefore~${Y(H) = 0}$ since~$\dim Y = \dim Z + 1$ and $Z = Y \cap H_2$.
	
	By definition of the reorientation operation, we thus obtain that~${Y_{-Z}(H_1) = +}$ and~$Y_{-Z}(H_2) = +$, while~$Y_{-Z}(H) = 0$ for all~${H \in \boundary(B) \setm \{H_1,H_2\}}$.
	In other words, $Y_{-Z}$ is the $2$-dimensional face of the base region~$B$ given by its intersection with all hyperplanes of~${\boundary(B) \setm \{H_1,H_2\}}$.
	Finally, since~$W$ is the ray of~$Y_{-Z}$ distinct from~$Z$, we obtain that~$W(H_1) = 0$, that~$W(H_2) = +$ and that~$W(H) = 0$ for all~${H \in \boundary(B) \setm \{H_1,H_2\}}$.
\end{proof}

\begin{lemma}
    \label{lemma:W-X_is_join_cand}
    We have~${X \fwcover -W}$ and $Y \fwle -W$ in the facial weak order.
\end{lemma}

\begin{proof}
	By \autoref{lem:covec_Y_Z_W_W-X_comparisons}, $W(H) \ge 0$ for all~$H \in \boundary(B)$, therefore $W(H) \ge 0$ for all~$H \in \arrangement$ by \autoref{cor:containment_B}.
	We therefore obtain that both $W$ and~$Z$ are rays of the base region~$B$, and thus~$Y_{-Z}$ is a $2$-dimensional face of~$B$ as well.
	
	Since~$X = \{0\}$ and $W$ is a ray of the base region~$B$, we have that~${-W(H) \le X(H)}$ for any~$H \in \arrangement$ so that~$X \fwle -W$. Since~$X \subseteq -W$ and $\dim (-W) - \dim(X) = 1$, we obtain that~${X \fwcover -W}$ by \autoref{prop:covers_for_fwo}.
	
	Assume now by contradiction that~$Y \not\fwle -W$.
	Then there exists~$H \in \arrangement$ such that~$Y(H) < -W(H)$.
	Since~$W(H) \ge 0$, it implies that~$Y(H) = -$ and $W(H) = 0$.
    But since~$Y_{-Z}(H) \ge 0$, we obtain by definition of reorientation that~${Z(H) = 0}$ and~${Y_{-Z}(H) = +}$.
	We conclude that~$W(H) = Z(H) = 0$ while~$Y_{-Z} = +$, contradicting the fact that~$Y_{-Z}$ is the $2$-dimensional face with rays~$W$ and~$Z$.
\end{proof}

\begin{lemma}
    \label{lemma:W-X_is_join}
    We have $X \fwjoin Y = -W$.
\end{lemma}

\begin{proof}
	Consider a face~$U$ of~$\faces$ such that~$X \fwle U$ and~$Y \fwle U$.
    We have that~${U(H) \le X(H) = 0}$ for all~$H \in \arrangement$ and, moreover,~$U(H_2) \le Y(H_2) = -$.
	Therefore, we obtain that~$U$ is a face of~$-B$ and~$-W \subseteq U$.
	We conclude that~$-W \fwle U$ by \autoref{cor:fwle_B}.
\end{proof}


\subsection{Further lattice properties of the facial weak order}

We end this section by describing some lattice properties of the facial weak order.
In particular we show that the lattice is self-dual, show the poset of regions is a sublattice, describe all the join-irreducible elements and show semidistributivity.


\subsubsection{Duality}
\label{subsubsec:duality}

Recall that the \definition{dual} of a lattice $(L, \le)$ is the order $(L, \le^{\op})$ where for $u,v \in L$, we have $u \le v$ if and only if $v \le^{\op} u$.
A lattice is \definition{self-dual} if it is isomorphic to its dual.
As with the poset of regions, the facial weak order is self-dual.
This follows from the fact that the poset of regions is itself self-dual and from the fact that the negative of every covector must also be in the set of covectors by the definition of oriented matroid.

\begin{proposition}
    \label{prop:selfdual_fwo}
    The map ${F \mapsto -F \coloneqq \set{-v}{v \in F}}$ is a self-duality of the facial weak order $\fw$.
\end{proposition}


\subsubsection{Sublattice}

In this subsection we show that if $\arrangement$ is simplicial, not only is~$\pr$ an induced subposet of~$\fw$ by \autoref{rem:subposet} and a lattice by \autoref{thm:simplicial_implies_lattice}, but it is in fact a sublattice of $\fw$.
Recall that a \definition{sublattice} $L'$ of a lattice~$L$ is an induced subposet such that~${u \join v \in L'}$ and $u \meet v \in L'$ for any~$u, v \in L'$.
The proof requires the following lemma which, just like the BEZ lemma, gives us a local way to verify if a subposet is a sublattice of a lattice, see \cite[Lemma 9-2.11]{Reading_LatticeTheoryChapter9}.
Recall that a poset is connected when the transitive closure of its comparability relation forms a single equivalence class.

\begin{lemma}
    \label{lem:sublattice_of_lattice}
    If~$P$ is a connected finite induced subposet of a lattice~$L$ such that~${x \join y \in P}$ for all $x, y, z \in P$ with $z \cover x$ and $z \cover y$, and $x \meet y \in P$ for all~${x, y, z \in P}$ with $x \cover z$ and $y \cover z$, then $P$ is a sublattice of $L$.
\end{lemma}

With this tool, we can prove the following statement.

\begin{proposition}
    \label{prop:sublattice}
    For a simplicial arrangement $\arrangement$, the lattice of regions is a sublattice of the facial weak order~$\fw$.
\end{proposition}

\begin{proof}
	\begin{figure}[b]
\DeclareDocumentCommand{\rs}{ O{1.1cm} O{->} m m O{0} O{0}} {
	\def \radius {#1}
	\def \inputPoints{#3}
	\def \excludeRoots{#4}
	\def \style {#2}
	\def \initialRotation {#5}
	\def \extraStyle {#6}

	\pgfmathtruncatemacro{\points}{\inputPoints * 2}
	\pgfmathsetmacro{\degrees}{360 / \points}
	
	\coordinate (0) at (0,0);
	
	\foreach \x in {1,...,\points}{%
		\pgfmathsetmacro{\location}{(\points+(\x-1))*\degrees + \initialRotation}
		
		\coordinate (\x) at (\location:\radius);
	}

	\ifthenelse{\equal{\excludeRoots}{}}{
		\foreach \x in {1,...,\points}{%
			\ifthenelse{\equal{\extraStyle}{0}}{
				\draw[\style] (0) -- (\x);
			}{
			\draw[\style, \extraStyle] (0) -- (\x);
		}
	}
}{
\foreach \x in {1,...,\points}{%
	\edef \showPoint {1};
	
	\foreach \y in \excludeRoots {
		\ifthenelse{\equal{\x}{\y}}{
			\xdef \showPoint {0};
		}{}
	}
	
	\ifthenelse{\equal{\showPoint}{1}}{
		\ifthenelse{\equal{\extraStyle}{0}}{
			\draw[\style] (0) -- (\x);
		}{
		\draw[\style, \extraStyle] (0) -- (\x);
	}
}{}
}
}  
}

\centerline{
    \begin{tikzpicture}
	    [
	    fdomain/.style={fill=blue!15!white,color=blue!15!white,opacity=0.3},
	    vdomain/.style={color=blue!85!white},
	    face/.style={draw=red!95!black,fill=red!95!black,color=red!95!black,ultra thick, -},
	    vertex/.style={inner sep=1pt,circle,draw=green!85!white,fill=green!85!white,thick},
	    vtext/.style={color=green!55!black},
	    ]
	    %
	    %
	    %
        \rs[3.5][ultra thick]{4}{}[22.5]
	    \node[above left] at (2) {$H_X$};
	    \node[above right] at (3) {$H_Y$};
        \rs[3][dashed]{4}{}[22.5]
	    %
	    %
	    \fill[fdomain] (0) -- (6) -- (7) -- cycle {};
        \node[vdomain] at (270:2.2) {$Z$};
	    \fill[fdomain] (0) -- (5) -- (6) -- cycle {};
	    \node[vdomain] at (225:2.2) {$X$};
        \fill[fdomain] (0) -- (7) -- (8) -- cycle {};
        \node[vdomain] at (315:2.2) {$Y$};
        \fill[fdomain] (0) -- (2) -- (3) -- cycle {};
        \node[vdomain] at (90:2.2) {$V$};
	    \node[] at (180:2.2) {$\vdots$};
	    \node[] at (0:2.2) {$\vdots$};
	    \node[vertex] at (0) {};
        \node[vtext] at (45:0.6) {$W$};
    \end{tikzpicture}
}
	    \caption{The construction of the join when $X$ and $Y$ are regions.}
		\label{fig:JoinSublattice}
	\end{figure}
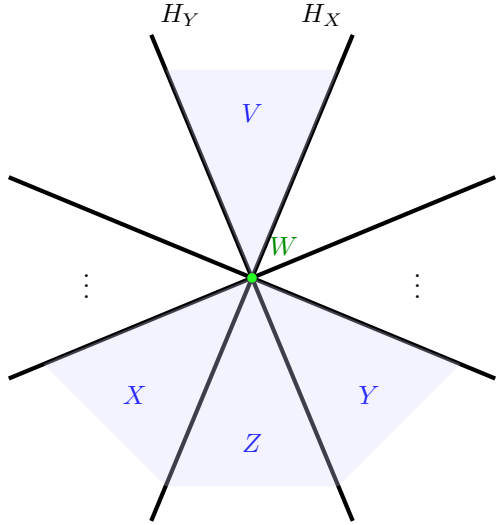
	
	By \autoref{rem:subposet}, $\pr$ is an induced subposet of~$\fw$.
	It is clearly connected as it contains the minimal and maximal elements of~$\fw$.
	Finally, by \autoref{prop:selfdual_fwo}, we just need to prove one of the two criteria of \autoref{lem:sublattice_of_lattice}.
    Consider thus three distinct regions $X, Y, Z \in \regions$ such that $Z \prcover X$ and $Z \prcover Y$.
    See \autoref{fig:JoinSublattice} for a (rank $2$) example.
    Since $Z \prcover X$, there is a hyperplane~$H_X$ separating $X$ and $Z$ such that $S(X) = S(Z) \cup \{H_X\}$.
    Similarly, there is a hyperplane~$H_Y$ separating~$Y$ and~$Z$ such that~$S(Y) = S(Z) \cup \{H_Y\}$.
    Since~$Z$ is simplicial, the face~$W \coloneqq Z \cap H_X \cap H_Y$ has codimension~$2$.
    We thus consider the rank~$2$ subarrangement~$\arrangement_W$.
    Since~$Z(H_X) = Z(H_Y) = +$, the face $Z_W$ is the base region of~$\arrangement_W$.
    Moreover, since~$X_W(H_X) = X(H_X) = -$ and~$Y_W(H_Y) = Y(H_Y) = -$, the join~$V$ of~$X_W$ and~$Y_W$ in~$\arrangement_W$ satisfies~$V(H_X) = V(H_Y) = -$ and is thus the opposite of the base region in~$\arrangement_W$.
    By \autoref{prop:look_inside_subarr}, the join of~$X$ and~$Y$ is the face~$U$ of~$\arrangement$ containing~$W$ and such that~$U_W = V$.
    We conclude that~$X \join Y$ is full dimensional, thus is a region.
    This concludes the proof by \autoref{lem:sublattice_of_lattice} and \autoref{prop:selfdual_fwo}.
\end{proof}


\subsubsection{Join-irreducible elements}

We next aim to find all the join-irreducible elements of the facial weak order.
An element $x$ of a finite lattice $L$ is \definition{join-irreducible} if $x \ne \bigvee L'$ for all $L' \subseteq L \setm \{x\}$.
Equivalently, $x$ is join-irreducible if and only if it covers exactly one element $x_\star$ of~$L$.
A \definition{meet-irreducible} element $y$ is defined in a similar manner where $y^\star$ is the unique element covering $y$.

For ease of notation, we denote by~$\ji(\FW)$ and~$\ji(\PR)$ (resp.~$\mi(\FW)$ and~$\mi(\PR)$) the sets of join-irreducible (resp.~meet-irreducible) elements in the facial weak order and in the poset of regions.

It turns out that the join-irreducible elements of the facial weak order are characterized by the join-irreducible elements of the poset of regions.
Each region~${R\in \ji(\PR)}$ gives a join-irreducible face $R$ in the facial weak order.
Additionally, the facet between $R$ and the unique region $R_\star$ it covers in the poset of regions is also a join-irreducible element in the facial weak order.
We give a small lemma before characterizing the join-irreducible elements in the facial weak order of a simplicial arrangement.

\begin{lemma}
    \label{lem:min_num_cov}
    Suppose $\arrangement$ is a simplicial hyperplane arrangement and $F$ a face of the arrangement.
    There exists exactly $\codim(F)$ facets of $F$ weakly below $F$ in the facial weak order.
\end{lemma}
\begin{proof} 
    If $F$ is a codimension $\codim(F)$ face then its span is the intersection of at least $\codim(F)$ hyperplanes.
    Since~$\arrangement$ is simplicial, exactly~$\codim(F)$ of these hyperplanes bound the base region of~$\arrangement_F$.
    Let $\mathcal{H}$ be this set of bounding hyperplanes.
    For each $H \in \mathcal{H}$ there exists a unique face $G$ such that $G(H) = +$ and $G(H') = 0$ for all $H' \in \mathcal{H} \setm \left\{ H \right\}$ since the base region must be simplicial by \autoref{prop:facial_int_are_simplicial}.
    In other words, there exists exactly $\codim(F)$ many codimension $\codim(F)-1$ faces covered by $F$ in the facial weak order.
\end{proof}

\begin{proposition}
    \label{prop:join_irr}
    Suppose $\arrangement$ is a simplicial hyperplane arrangement and let $F$ be a face with associated facial interval $[m_F, M_F]$.
    Then~${F \in \ji(\FW)}$ if and only if~${M_F \in \ji(\PR)}$ and~${\codim(F) \in \left\{ 0,1 \right\}}$.
\end{proposition}

\begin{proof}
    We first suppose that $F$ is join-irreducible in $\fw$.
    Since a join-irre\-ducible element can cover at most one element, \autoref{lem:min_num_cov} implies ${\codim(F) \leq 1}$.

    Suppose first that $\codim(F) = 0$.
    Then $F$ is a region and $m_F = M_F = F$.
    Let~$F_\star$ be the unique face covered by $F$.
    By \autoref{cor:fwo_covers},~${\abs{\dim(F) - \dim(F_\star)} = 1}$ and therefore~${\codim(F_\star) = 1}$.
    Therefore, there exists a unique hyperplane $H$ bounding~$M_F$ such that $H \in S(M_F)$ and $H \cap M_F = F_\star$.
    Thus, there is a unique region~$R$ such that $S(R) = S(M_F) \setm \left\{ H \right\}$.
    In other words, $M_F \in \ji(\PR)$.

    Suppose next that $\codim(F) = 1$.
    Again by \autoref{cor:fwo_covers}  only codimension~$0$ and codimension~$2$ faces can be covered by $F$ in the facial weak order.
    By \autoref{lem:min_num_cov} there exists at least one codimension~$0$ face covered by $F$.
    Therefore $F_\star$ is a region and, as $F$ is join-irreducible, $F$ does not cover any codimension~$2$ face.
    If contrarily $M_F \notin \ji(\PR)$ then there exists a boundary hyperplane $H$ of $M_F$ such that $H \cap M_F \ne F$.
    Let $G = H \cap M_F$.
    Then $G \cap F$ is a face with codimension $2$ such that $G \cap F \subseteq F$ and $M_F = M_G$.
    Thus $G \cap F$ is a codimension $2$ face covered by $F$, a contradiction.

    To show the other direction, we conversely suppose that $M_F \in \ji(\PR)$ and $\codim(F) \in \{0,1\}$.
    Since $M_F \in \ji(\PR)$ it covers the unique region ${M_F}_\star$ and there is a unique face $G$ between the two regions with facial interval $[{M_F}_\star,M_F]$.
    If~${\codim(F) = 0}$ then $F = M_F$ and, since only codimension $1$ faces can be covered by $F$, then $G$ is the unique facet of $F$ which is covered by~$F$, \ie $F \in \ji(\FW)$.
    If $\codim(F) = 1$ then $F = G$ and ${M_F}_\star \fwcover F$ by construction.
    To prove $F$ doesn't cover another face it suffices to observe that if there was another face $G'$ covered by $F$ it must be of codimension $2$ by \autoref{lem:min_num_cov}.
    But then, $M_F \cap G' = G'$ since~${G' \subseteq F \subseteq M_F}$.
    In other words, there exists a second facet to $M_F$ weakly below~$M_F$ by simpliciality, a contradiction.
\end{proof}

As we saw previously, these join-irreducibles come in pairs.
This comes from introducing the edges of the poset of regions as vertices in the facial weak order.

\begin{corollary}
    \label{cor:join_irr_pair}
    Let $F$ and $F'$ be faces of codimension $1$ and $0$ respectively such that $F \fwcover F'$.
    Then $F$ is join-irreducible in the facial weak order if and only if~$F'$ is join-irreducible in the facial weak order.
\end{corollary}

\begin{proof}
    If $F$ is a codimension $1$ face then there exists a unique codimension $0$ face covering it.
    Conversely, for every codimension $0$ face (excluding the base region) there is at least one codimension $1$ face covered by it.
    In other words, $F$ exists if and only if~$F'$ exists (where $F'$ is not the base region).
    Furthermore, $F \fwcover F'$ implies $F$ is the face strictly below the region $F'$.
    In other words, $M_F = M_{F'}$.
    But this implies $M_F \in \ji(\PR)$ if and only if $M_{F'} \in \ji(\PR)$.
    Since $\codim(F) = 1$ and $\codim(F') = 0$ this implies $F$ is join-irreducible if and only if~$F'$ is join-irreducible.
\end{proof}

Recalling that our lattice is self-dual by \autoref{prop:selfdual_fwo} we have the following two corollaries.

\begin{corollary}
    \label{cor:meet_irr}
    Suppose $\arrangement$ is a simplicial arrangement and let $F$ be a face with associated facial interval $[m_F, M_F]$.
    Then $F \in \mi(\FW)$ if and only if~${m_F \in \mi(\PR)}$ and $\codim(F) \in \left\{ 0,1 \right\}$.
\end{corollary}

\begin{corollary}
    \label{cor:meet_irr_pair}
    Let $F$ and $F'$ be faces of codimension $0$ and $1$ respectively such that $F \fwcover F'$.
    Then $F$ is meet-irreducible in the facial weak order if and only if~$F'$ is meet-irreducible in the facial weak order.
\end{corollary}


\subsubsection{Semidistributivity}

In this subsection, we show that our lattice is semidistributive.
A lattice is \definition{join-semidistributive} if $x \join y = x \join z$ implies $x \join y = x \join (y \meet z)$.
Similarly, a lattice is \definition{meet-semidistributive} if the dual condition holds.
A lattice is \definition{semidistributive} if it is both meet-semidistributive and join-semidistributive.

Recall that for a join-irreducible element $x$, the unique element it covers is denoted by $x_\star$, \ie $x_\star \cover x$.
Likewise, for a meet-irreducible element $y$, the unique element covered by it is denoted $y^\star$, \ie $y \cover y^\star$.
Given a join-irreducible element $x$ and a meet-irreducible element $y$ for a finite lattice $L$, we say that $(x_\star,x)$ and~$(y,y^\star)$ are \definition{perspective} if $x \meet y = x_\star$ and $x \join y = y^\star$. 
We have the following lemma, see \cite[Theorem 2.56]{FreeseJezekNation_FreeLattices}.

\begin{lemma}
    \label{lem:m_semidist_iff_subcrit}
    A finite lattice is meet-semidistributive if and only if for every join-irreducible element $x$ there exists a unique meet-irreducible element $y$ such that~${(x_\star, x)}$ and $(y, y^\star)$ are perspective.
\end{lemma}

We will use the following theorem, see \cite[Theorem 3]{Reading_LatticePropertiesPosetRegions}.

\begin{theorem}
    \label{thm:poc_semidist}
    For a simplicial arrangement $\arrangement$, its poset of regions is a semidistributive lattice.
\end{theorem}

It turns out that due to the self-duality of the facial weak order and its intimate connection with the poset of regions that perspective pairs do exist.
This will give us that the facial weak order is semidistributive.
Recall that $\ji(\FW)$ ($\ji(\PR)$) and $\mi(\FW)$ ($\mi(\PR)$) are the sets of join-irreducible and meet-irreducible elements in the facial weak order (poset of regions) respectively.

\begin{proposition}
    \label{prop:meetSD_fwo}
    The facial weak order is meet-semi-distributive.
\end{proposition}

\begin{proof}
    By \autoref{lem:m_semidist_iff_subcrit} it suffices to show for every face $F \in \ji(\FW)$ there exists a unique face~${G \in \mi(\FW)}$ such that $(F_\star, F)$ and $(G, G^\star)$ are perspective, \ie~${F \fwmeet G = F_\star}$ and~${F \fwjoin G = G^\star}$.
    An example can be seen in \autoref{fig:MSD}.
    Let $[m_F, M_F]$ denote the facial interval of $F$.
    Since~${F \in \ji(\FW)}$, by \autoref{prop:join_irr},~${M_F \in \ji(\PR)}$ and~${\codim{F} \in \{0,1\}}$.

    Suppose $\codim(F) = 1$.
    Since $F \in \ji(\FW)$ by \autoref{lem:min_num_cov} there exists a unique face $F'$ of codimension $0$ such that $F \fwcover F'$ and $M_F = M_{F'}$.
    By \autoref{cor:join_irr_pair} $F' \in \ji(\FW)$ with $F = F'_\star$.
    We then have the following chain of covers~${F_\star \fwcover F = F'_\star \fwcover F'}$.
    
    Since $M_{F'} = M_{F} \in \ji(\PR)$, by \autoref{thm:poc_semidist} and \autoref{lem:m_semidist_iff_subcrit}, there exists a unique meet-irreducible region $M_{G}$ such that $\left(({M_{F'}})_\star,M_{F'}\right)$ and $\left(M_{G},(M_{G})^\star\right)$  are perspective in the poset of regions.
    Let $G = M_{G}$ be the codimension $0$ face associated to the region $M_{G}$.
    Since $M_{G} = m_{G}$ is meet-irreducible in the poset of regions, then~$G$ is meet-irreducible in the facial weak order by \autoref{cor:meet_irr} since it is of codimension $0$.
    Then, by definition of meet-irreducible, there exists a unique face~$G'$ of codimension~$1$ such that~${G \fwcover G'}$.
    Furthermore, by \autoref{cor:meet_irr_pair},~$G'$ is meet-irreducible in the facial weak order with $m_{G} = m_{G'}$.
    We then have the following chain of covers $G \fwcover G^\star = G' \fwcover {G'}^\star$.

    Recalling that $\left((M_{F'})_\star,M_{F'}\right)$ and $\left(M_{G}, (M_{G})^\star\right)$ are perspective in the poset of regions, and furthermore, since $M_{F'} = M_F$ and $M_{G} = m_{G} = m_{G'}$ we have:
    \[
        M_F \prmeet m_{G'} = m_F \qquad M_F \prjoin m_{G'} = M_{G'}.\tag{$\diamond$}\label{ex:sd_por_joins}
    \]

    This implies that the pair $(F_\star, F)$ and $(G,G^\star)$ and the pair $(F'_\star, F')$ and $(G', G'^\star)$ are both perspective.
    Indeed, looking at the first case $(F_\star, F)$ and $(G,G^\star)$, we want to show~${F \fwmeet G = F_\star}$ and~${F \fwjoin G = {G}^\star = G'}$.
    For $F \fwmeet G = F_\star$, since~$F$ covers only~$F_\star$ by definition of join-irreducible, it suffices to show~${F_\star \fwle G}$.
    Similarly, since $G$ is only covered by~${{G}^\star = G'}$, to show $F \fwjoin G = G'$, it suffices to show $F \fwle G'$.
    To show~${F \fwle G'}$ it suffices to observe that $m_F \prle m_{G'}$ and~${M_F \prle M_{G'}}$.
    Indeed, by~\autoref{ex:sd_por_joins}, we have~${M_F \prmeet m_{G'} = m_F}$, implying~${m_F \prle m_{G'}}$ and~${M_F \prjoin m_{G'} = M_{G'}}$ giving~${M_F \prle M_{G'}}$ as desired.
    To show~${F_\star \fwle G}$ we follow a similar approach by proving that $m_{F_\star} \prle m_{G}$ and $M_{F_\star} \prle M_{G}$.
    Since~${M_F \prmeet m_{G} = M_F \prmeet m_{G'} = m_F = m_{F_\star}}$, therefore $m_{F_\star} \prle m_{G}$.
    Also, since $G$ and $F_\star$ are of codimension $0$ we have~${M_{F_\star} = m_{F_\star} \prle m_{G} = M_{G}}$ as desired.
    Therefore $F_\star \fwle G$ in the facial weak order.

    The case $({F'}_\star, F')$ and $(G',G'^\star)$ is handled similarly.

    Notice that the case where $\codim(F) = 0$ was handled in the proof above since~$F'$ is a join-irreducible element in the facial weak order with codimension $0$.
\begin{figure}[b]
	\centerline{
    \begin{tikzpicture}
	    [
	    vertex/.style={inner sep=1pt,circle,draw=green!85!white,fill=green!85!white,thick},
	    ]
	    %
	    %
        \coordinate (F1) at (0,0);
        \coordinate (F2) at (0,1);
        \coordinate (G) at (0,2);
        \coordinate (G') at (1,1);
        \coordinate (F1') at (1,2);
        \coordinate (F2') at (1,3);
        \draw (F1) -- (F2);
        \draw (F2) -- (G);
        \draw (G') -- (F1');
        \draw (F1') -- (F2');
        \draw[dashed] (F1) -- (G');
        \draw[dashed] (F2) -- (F1');
        \draw[dashed] (G) -- (F2');
	    %
        \node[left] at (F1) {$m_{F} = F_\star$};
        \node[left] at (F2) {$F$};
        \node[left] at (G) {$M_{F} = F'$};
        \node[right] at (F1') {$G'$};
        \node[right] at (F2') {$M_{G'} = {G'}^\star$};
        \node[right] at (G') {$m_{G'} = G$};
    \end{tikzpicture}
}
	\caption{Meet-semidistributivity in the facial weak order.}
	\label{fig:MSD}
\end{figure}
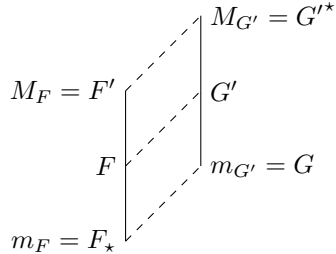
\end{proof}

Combining this proposition with the fact that the lattice is self-dual, we get join-semidistributivity for free.
In particular, we get that our lattice is semidistributive.

\begin{theorem}
    For a simplicial arrangement $\arrangement$, the facial weak order is semidistributive.
\end{theorem}


\section{Topology of the facial weak order}
\label{sec:topology}

In this section, we determine the homotopy type of intervals of the facial weak order; see \autoref{thm_weak_top}. Before proving this theorem, some preliminary results on the topology of posets are given in \autoref{subsec:prelim}.


\subsection{Poset topology}
\label{subsec:prelim}

In this section, we recall some standard tools and definitions concerning simplicial complexes that we use in the proof of \autoref{thm_weak_top}. The main result we need is \autoref{lem_link_del}.

An \definition{abstract simplicial complex $\Delta$} is a ground set $E$ with a collection $\Delta$ of subsets of $E$ (called \definition{faces}) such that if $F \in \Delta$ and $G \subseteq F$ then $G \in \Delta$.
The \definition{dimension of a face} $F \in \Delta$ is given by $\dim(F) = \abs{F} - 1$ and the \definition{dimension of $\Delta$} is the maximal dimension of all faces.
An abstract simplicial complex of dimension $d$ can be realized geometrically in $\R^{2d +1}$ by a union of simplices well-defined up to homeomorphism.
We denote this geometrical realization by $\norm{\Delta}$.
The \definition{deletion}~$\del_{\Delta}(F)$ of $F$ from $\Delta$ is the subcomplex of faces disjoint from $F$.
The \definition{link} $\lk_{\Delta}(F)$ of $F$ is the subcomplex of faces $G$ for which $F \cap G = \varnothing$ and $F\cup G$ is a face of~$\Delta$.
The \definition{join} $\Delta*\Delta^{\prime}$ of two complexes with disjoint ground sets is the simplicial complex with faces $F \sqcup F^{\prime}$ where $F \in \Delta,\ F^{\prime} \in \Delta^{\prime}$.
The \definition{cone} $\{v\}*\Delta$ is the join of $\Delta$ with a one-element complex.
The \definition{suspension} $\susp\Delta$ is the join of $\Delta$ with a discrete two-element complex.
A fundamental homotopy equivalence connecting the deletion and link to the original complex is the following, which can be proved using the Carrier Lemma (\eg \cite[Lemma 2.1]{Walker_homotopy}).

~\\

\begin{lemma}
	\label{lem_link_del}
	Let $F$ be a face of a simplicial complex $\Delta$.
	\begin{enumerate}
		\item If $\lk_{\Delta}(F)$ is contractible, then $\Delta$ is homotopy equivalent to $\del_{\Delta}(F)$.
		\item If $\del_{\Delta}(F)$ is contractible, then $\Delta$ is homotopy equivalent to the suspension of $\lk_{\Delta}(F)$.
	\end{enumerate}
\end{lemma}

Given elements $x,y$ of a poset $P$, the \definition{open interval} $(x,y)$ (resp.~\definition{closed interval}~$[x,y]$) is the set of $z \in P$ such that $x < z < y$ (resp.~$x \leq z \leq y$). We let $P_{<x}$ (resp.~$P_{>x}$) denote the set of elements $y\in P$ such that $y < x$ (resp.~$y > x$). The \definition{order complex} $\Delta(P)$ of a poset $P$ is the simplicial complex of chains $x_0 < \cdots < x_d$ of elements of $P$. The link of a face $x_0 < \cdots < x_d$ is isomorphic to the join of the order complexes of $P_{<x_0}, (x_0, x_1), \ldots, (x_{d-1}, x_d), P_{>x_d}$. Hence, the local topology of $\Delta(P)$ is completely determined by the topology of open intervals and principal order ideals and filters of $P$. In the remainder of this section, whenever we write about the topology of $P$, we mean the topology of its order complex.

In \autoref{subsec:topology}, we explicitly determine the homotopy types of intervals of the facial weak order. To this end, we will use some consequences of \autoref{lem_link_del}.

\begin{lemma}
	\label{lem_poset_reduction}
    Let $P$ be a poset, and let $X \subseteq P$ such that $P_{<x}$ is contractible for all~${x \in X}$. Then $P$ is homotopy equivalent to~${P \setm X}$.
\end{lemma}

\begin{proof}
	Let $X = \{x_1, x_2, \ldots, x_n\}$ so that whenever $i \leq j$, we have $x_i \nless x_j$. We claim that $P$ is homotopy equivalent to $P \setm \{x_1, \ldots, x_i\}$ for any $i$. 
    First observe that~${\lk_{\Delta(P)}(x_1) = \Delta(P_{<x_1})*\Delta(P_{>x_1})}$ is contractible, so $P$ is homotopy equivalent to~${P \setm \{x_1\}}$. More generally, from the assumption on the ordering of elements of $X$, we have $\lk_{\Delta(P \setm  \{x_1, \ldots, x_{i-1}\})}(x_i) = \Delta(P_{<x_i})*\Delta^{\prime}$ for some simplicial complex $\Delta^{\prime}$. 
	This is again contractible, so by induction, $P$ is homotopy equivalent to $P \setm \{x_1, \ldots, x_i\}$. Taking $i=n$, we have completed the proof.
\end{proof}

A \definition{closure operator} (resp.~\definition{dual closure operator}) on a poset $P$ is an idempotent, order preserving, increasing (resp.~decreasing) function $f:P\rightarrow P$.

\begin{lemma}[Corollary 10.12 \cite{Bjorner_topologicalMethods}]
	\label{lem_closure_top}
	If $f:P\rightarrow P$ is a closure operator or a dual closure operator, then $P$ is homotopy equivalent to $f(P)$.
\end{lemma}

This lemma may be proved in many ways, \eg by repeated application of \autoref{lem_link_del} as we did for \autoref{lem_poset_reduction} or by application of Quillen's Fiber Lemma \cite[Theorem 10.5]{Bjorner_topologicalMethods}.


\subsection{Topology of intervals of the facial weak order}
\label{subsec:topology}

Let $\arrangement$ be a real, central hyperplane arrangement with base region $B$. As usual, we orient the hyperplanes in $\arrangement$ so that
$$B=\bigcap_{H \in \arrangement}H^+.$$

Recall that the \definition{poset of regions} $\pr$ is the set of regions with the partial order $R \prle R^{\prime}$ if and only if $S(R) \subseteq S(R^{\prime})$.
With this ordering, $B$ is the unique minimum element of the poset of regions. Given faces $X,Y$ of $\arrangement$, we say that $X$ is \definition{incident} to $Y$ if $X \supseteq Y$.

P.~H.~Edelman and J.~W.~Walker determined the local topology of the poset of regions \cite{EdelmanWalker_homotopy}. As this result will be used in the proof of \autoref{thm_weak_top}, we state it here.

\begin{theorem}[\cite{EdelmanWalker_homotopy}]
	\label{thm_edelman_walker}
    For each face $X \in \faces$, the set of regions incident to~$X$ is an interval $[R_1,R_2]_{\prle}$ of the poset of regions such that the open interval~${(R_1,R_2)_{\prle}}$ is homotopy equivalent to a sphere of dimension $\codim(X)-2$. Every other interval is contractible.
\end{theorem}

Recall that $\vspan(X)$ denotes the subspace spanned by a face $X$.
The \definition{poset of intersection subspaces}, or simply \definition{intersection poset}, is the poset on the subspaces~${L(\arrangement) = \set{\bigcap_{H \in \intersection} H}{\intersection \subseteq \arrangement}}$ ordered by reverse inclusion.
As before this poset is a lattice when the vector space $V$ is added as the bottom element and is called the \definition{intersection lattice}.
For $X$ and $Y$ in $L(\arrangement)$, the join is given by~${X \join_L Y = X \cap Y}$ and the meet by $ X \meet_L Y = \bigcap_{ {X \cup Y \subseteq Z} } Z$.
For further information on the intersection lattice we refer the reader to P.~Orlik and H.~Terao's book \cite[Section 2]{OrlikTerao}.
For a face $X$, let $\arrangement^X$ denote the \definition{restriction} of $\arrangement$ to $\vspan(X)$ where
\[
    \arrangement^X = \set{H \cap \vspan(X)}{H \in \arrangement \setm \arrangement_X}.
\]
The set $\arrangement^X$ is an arrangement of hyperplanes in the vector space spanned by $X$.
For a covector $X$, we recall the map $\pi_X: \faces \rightarrow \faces[\arrangement_X]$ where for any covector~$Y$ of~$\arrangement$, its image is the covector with $\pi_X(Y)(H) = Y(H)$ for $H \in \arrangement_X$.
Similarly, one can define a map $\iota^X: \faces[\arrangement^X] \rightarrow \faces$ such that for a covector $Y$ of $\arrangement^X$ we have
\[
    \iota^X(Y)(H) = \begin{cases} Y(H) \ \ & \text{if } H \in \arrangement^X \\ 0 \ \ & \text{if } H \in \arrangement_X \end{cases}.
\]

It is clear that $\iota^X$ is injective, so if $Y$ is a face of $\arrangement$ contained in $\vspan(X)$, we write~${(\iota^X)^{-1}(Y)}$ for the corresponding face of $\arrangement^X$.
To simplify notation, we define~$Y^X$ to be $(\iota^X)^{-1}(Y)$ in this case.


\begin{lemma}
	\label{lem_restriction}
    Let $X,Y$ be covectors such that $X \fwle Y$ in $\fw$, and let $Z$ be a covector such that $\vspan(Z) = \vspan(X) \meet_L \vspan(Y)$.
	Then the interval $[X,Y]$ of $\fw$ is isomorphic to the interval $[X^Z,Y^Z]$ of $\fw[\arrangement^Z][X^Z\circ(-Y^Z)]$.
\end{lemma}

\begin{proof}
	For $H \in \arrangement$ such that $X(H)=0=Y(H)$, any covector $W$ in the interval~$[X,Y]$ must satisfy $W(H)=0$.
    Consequently,~${\vspan(W) \subseteq \vspan(Z)}$ for~${W\in[X,Y]}$, so the restriction $W^Z$ is a covector of $\arrangement^Z$.
    Moreover, the map~${[X,Y] \rightarrow \covectors(\arrangement^Z)}$ is a bijection onto its image.
	Since $\vspan(X^Z) \meet_L \vspan(Y^Z) = \vspan(Z)$ in the intersection lattice, the concatenation $X^Z \circ (-Y^Z)$ is a region of~$\arrangement^Z$.
	Moreover, for $H \in \arrangement^Z$, we have $X^Z \circ (-Y^Z)(H) = X^Z(H)$ if $X^Z(H) \ne 0$, and $X^Z \circ (-Y^Z)(H)=+$ otherwise.
    Hence, $X^Z \fwle Y^Z$ in $\fw[\arrangement^Z][X^Z \circ (-Y^Z)]$, and if $W \in [X,Y]$, then~${W^Z \in [X^Z,Y^Z]}$. Conversely, every element of $[X^Z,Y^Z]$ is the restriction of some covector in $[X,Y]$.
\end{proof}

We are now ready to prove the main theorem. We make use of the fact that the proper part of the face lattice of a polyhedral cone of dimension $d$ is homeomorphic to a sphere of dimension $d-2$.

\begin{theorem}
	\label{thm_weak_top}
	Let $\arrangement$ be an arrangement with base region $B$. Let $X,Y$ be covectors such that $X \fwle Y$ and set $Z = X \cap Y$. 
	If $X \fwle Z \fwle Y$ and ${Z = X_{-Z} \cap Y}$, then the order complex of the open interval $(X,Y)$ in $\fw$ is homotopy equivalent to a sphere of dimension $\dim(X)+\dim(Y)-2\dim(Z)-2$. 
	Every other interval is contractible.
\end{theorem}

\begin{proof}
	Let $X,Y \in \fw$ such that $X \fwle Y$ and set $Z = X \cap Y$. 
	Let $Q$ be the open interval $(X,Y)$ in the facial weak order. 
	We determine the topology of $Q$.

	If $Z=X$ then $Q$ is an interval in the face lattice, so it is homeomorphic to a sphere of dimension $\dim(Y)-\dim(X)-2$, as desired. 
	Similarly, if $Z=Y$, then $Q$ is homeomorphic to a sphere of dimension $\dim(X)-\dim(Y)-2$. Hence, we may assume $X,Y,Z$ are all distinct.

	By \autoref{lem:localization}, we may assume that $Z=0$ since the poset $Q = (X,Y)$ is isomorphic to $(X_Z,Y_Z)$.
	Hence, we write $-X$ for the covector $X_{-Z}$.
	By \autoref{lem_restriction}, we may assume that $\vspan(X) \meet_L \vspan(Y) = V$ in the intersection lattice.
    In particular,~${X \circ Y}$ and $Y \circ X$ are regions. We will make these assumptions for most of the proof unless indicated otherwise.

	Assume $Z \in Q$ and $Z = (-X) \cap Y$ both hold, and let $\Delta$ be the order complex of $Q$. We prove that $\del_{\Delta}(\{Z\})$ is contractible by induction on $\dim(Y)$.

    Let $L_{>Y}$ denote the set of faces strictly less than the face~$Y$ in the face lattice, \ie~${L_{>Y} = \set{Z}{Z >_L Y} = \set{Z}{Z \subsetneq Y}}$.
	Applying the inductive hypothesis with \autoref{lem_poset_reduction}, the poset $Q \setm \{Z\}$ is homotopy equivalent to $Q \setm L_{>Y}$. We note that this statement is vacuously true if $\dim(Y) = 1$. Set $P = Q \setm L_{>Y}$. 
	Define a map~$f$ on the closed interval $[X,Y]$ of the facial weak order, where $f(W) = W \circ Y$. 
	This is well-defined by \autoref{lem:F_leq_G_=>_F_leq_FcG_leq_GcF_leq_G}.
	We claim that $f$ is a closure operator.
    It is clear that~$f$ is idempotent by properties of composition. 
	Since $Z \fwl Y$ in the facial weak order, every entry of $Y$ is either $0$ or $-$. 
	Hence,~$f$ can only change some $0$ entries of $w$ to $-$, so it is order preserving and increasing. 
	Since $W \circ Y \subseteq Y$ only if~$W$ is a face of $Y$, the operator $f$ restricts to $P$. 
    \autoref{lem_closure_top} implies that~${P \simeq f(P)}$.

	Now define $g$ on $[X,Y]$ where $g(W)=W\circ X$. 
	This map is a dual closure operator. 
	Assume that $W \in Q$ such that $g \big( f(W) \big) = X$. 
	Then $f(W)$ must be a face of $X$. 
	Since $W$ is a face of $f(W)$, we deduce that $W$ is a face of $X$. 
	The set of faces of $X$ intersected with $[X,Y]$ is an order ideal of $[X,Y]$ in the facial weak order. 
	Since $(-X) \cap Y = 0$, the composite $X \circ Y$ is a region distinct from~$X$. 
    Then~${X \circ Y \fwle W \circ Y}$, so $W \circ Y$ is not a face of $X$. 
	This is a contradiction. 
	Hence, $g$ restricts to $f(P)$, and we conclude that $P \simeq g \big( f(P) \big)$.

	Since $X$ and $Y$ have disjoint supports, the composite $Y\circ X$ is a region. 
	Hence, the image of $g\circ f$ is the set of regions in $Q$. 
	This set of regions has a maximum element, namely $Y\circ X$. 
	Hence, it is contractible, as desired.

	Since $\del_{\Delta}(\{Z\})$ is contractible, we conclude that $\Delta$ is homotopy equivalent to the suspension of $\lk_{\Delta}(\{Z\})$ by \autoref{lem_link_del}. 
    By definition,~${\lk_{\Delta}(\{Z\}) = \Delta((X,Z))*\Delta((Z,Y))}$.
	But $\Delta((X,Z))$ (resp.~$\Delta((Z,Y))$) is the order complex of the proper part of the face lattice of the cone $X$ (resp.~$Y$). 
	Hence, $\Delta((X,Z))$ is homeomorphic to $\Sbb^{\dim(X)-\dim(Z)-2}$.
	Since $\Sbb^p*\Sbb^q\cong\Sbb^{p+q+1}$ and $\susp(\Sbb^p) \simeq \Sbb^{p+1}$, we have
	\begin{align*}
		\Delta
		& \simeq \susp \big( \lk_{\Delta}(\{Z\}) \big) \\
		& \simeq \susp \big( \Delta((X,Z))*\Delta((Z,Y)) \big) \\
		& \simeq \susp \big( \Sbb^{\dim(X)-\dim(Z)-2}*\Sbb^{\dim(Y)-rk(Z)-2} \big) \\
		& \simeq \Sbb^{\dim(X)+\dim(Y)-2\dim(Z)-2}
	\end{align*}

	Now assume that $Z\notin Q$. We prove that $Q$ is contractible.

	Since $Z$ is not between $X$ and $Y$, there exists $H\in\arrangement$ such that $Z(H)=0$ and either $X(H)=Y(H)=-$ or $X(H)=Y(H)=+$. 
	Replacing $B$ with $-B$, we may assume without loss of generality that $X(H)=Y(H)=-$ and $Z(H)=0$. 
	If~$W$ is any face of $Y$ with $W \fwle Y$, then $W(H)=-$. 
	But $(W \cap X) \subseteq Z$, so $W \cap X$ is not between $X$ and $W$. 
	By induction, $Q$ is homotopy equivalent to $Q \setm L_{>Y}$. 
	Let $P = Q \setm L_{>Y}$. As before, we consider operators $f$ and $g$ on $[X,Y]$. 
	These two operators again restrict to $P$, and $g \big( f(P) \big)$ is the subposet of regions in $Q$. 
	If~$Y$ is not a region, then $Y \circ X$ is the unique maximum element of $g \big( f(P) \big)$. 
	If~$X$ is not a region, then $X \circ Y$ is the unique minimum element of $g \big( f(P) \big)$. 
	In either case, the interval $Q=(X,Y)$ is contractible. 
	If both $X$ and $Y$ are regions, then~$Q$ is contractible by \autoref{thm_edelman_walker} since $g \big( f(P) \big)$ is an open interval of the poset of regions that is not facial.

	Finally, assume that $Z \in Q$ but $Z \ne X_{-Z} \cap Y$. We prove that $Q$ is contractible.

	Assume $X$ is not a region.
	Then $Y \cap (-X)$ is a proper face of $Y$, as otherwise there would exist a hyperplane $H\in\arrangement$ containing both $X$ and $Y$.
    Let
    \[
        {P = Q \setm \set{W \in L_{>Y}}{(-X) \cap W \ne Z}}
        \text{ and } P_{>Z} = \set{W \in P}{W > Z}.
    \]
	By induction, $P$ is homotopy equivalent to $Q$.
	Then~$P_{>Z}$ is contractible since $L_{>Y}\setm\{Z\}$ is the proper part of the face lattice of the cone $Y$, and $P_{>Z}$ is the deletion of some face from this sphere.

	Consequently, $P \simeq P \setm \{Z\}$. 
	We prove that $P \setm \{Z\}$ is contractible by induction on $\dim(Y)$. 
    We have already proved that $(X,Y^{\prime})\setm\{Z\}$ is contractible for~${Y^{\prime} \in L_{>Y} \cap P \setm \{Z\}}$. 
	Hence, $P \setm \{Z\} \simeq P \setm L_{>Y}$. 
    Set $P^{\prime} = P \setm L_{>Y}$. Using the operators $f$ and $g$ from before, we deduce that $Q$ is homotopy equivalent to~${g \big( f(P^{\prime}) \big)}$.
    Since $X$ is not a region, $g \big( f(P^{\prime}) \big)$ has a minimum element, namely~${X \circ Y}$.
	Hence, it is contractible.

	If $X$ is a region but $Y$ is not a region, then a dual argument shows that $Q$ is contractible. 
	Hence, we may assume both $X$ and $Y$ are regions. 
	Since this is the last remaining case, we deduce that for $W\in Q$ the interval $(X,W)$ is not contractible if and only if $W$ is an upper face of $X$. 
	Hence, $Q$ is homotopy equivalent to $L_{>X}$. 
	This set of covectors has a maximum element in $Q$, namely $W=Z$. 
	Hence, $Q$ is contractible.
\end{proof}


\subsection{M\"obius function}

Recall that the \definition{M\"obius function} of a poset $P$ is the function $\mu: P \times P \to \Z$ defined inductively by
\[
    \mu(x,y) = \begin{cases}
        1 & \text{if } x = y,\\
        -\sum_{x \leq z < y} \mu(x,z) & \text{if } x < y,\\
        0 & \text{otherwise.}
    \end{cases}
\]
For more information on the M\"obius function we refer the reader to \cite{Stanley_EnumerativeCombinatorics}.

We recall that the M\"obius function can be restated using its homotopy type.
In fact, $\mu(x,y) + 1 = \sum (-1)^i \rk H_i(\Delta( (x,y)))$ where $H_i(\Delta( (x,y)))$ is the simplicial $i$th homology group and $\Delta( (x,y))$ is the order complex for the open interval $(x,y)$.
The rank of the $i$th homology group is sometimes referred to as the \definition{$i$th Betti number}.

Recall further that a contractible interval $(x,y)$ has trivial homology (homotopy equivalent to a point).
Thus $H_0(\Delta( (x,y)))  \iso \Z$ and~${H_i(\Delta( (x,y))) \iso \{0\}}$ for all~${i > 0}$, \ie $\sum (-1)^i \rk H_i(\Delta( (x,y))) = 1$.
Therefore we have $\mu(x,y) = 0$.
Additionally, recall that a sphere $\Sbb^n$ has homology $H_0(\Sbb^n) \iso \Z$, $H_n(\Sbb^n) \iso \Z$ and~${H_i(\Sbb^n) \iso \{0\}}$ for $0 < i < n$, \ie $\sum (-1)^i \rk H_i(\Delta( (x,y))) = 1 + (-1)^n$.
Therefore if our interval $(x,y)$ is homotopy equivalent to $\Sbb^n$ we have $\mu(x,y) = (-1)^n$.
For more information on how the M\"obius function relates to homology we refer the reader to the book by Stanley \cite{Stanley_EnumerativeCombinatorics}, the book by Munkres \cite{Munkres_ElementsOfAlgebraicToplogy}, or the chapter by Bj\"orner \cite{Bjorner_topologicalMethods}.

As a consequence to \autoref{thm_weak_top} we have the following corollary.

\begin{corollary}
    Let $\arrangement$ be an arrangement with base region $B$.
    Let $X,Y$ be covectors such that $X \fwle Y$ and set $Z = X \cap Y$. 
    \[
        \mu(X,Y) = \begin{cases}
            (-1)^{\dim(X) + \dim(Y)} & X \fwle Z \fwle Y \text{ and } Z = X_{-Z} \cap Y\\
            0 & \text{otherwise.}
        \end{cases}
    \]
\end{corollary}

\section*{Acknowledgements}
    We thank Nathan Reading and Hugh Thomas for their relevant suggestions on an initial version of this paper.


\nocite{*}
\bibliographystyle{alpha}
\bibliography{paper}
\label{sec:biblio}

\end{document}